\documentclass[a4paper,12pt]{article}

\usepackage{a4wide}
\usepackage{amsmath,amsfonts,amssymb,amsgen,amsbsy,latexsym,amsthm}
\usepackage{graphicx}
\usepackage{float,here}
\usepackage{xcolor}
\usepackage{dsfont}
\usepackage{mathrsfs}
\usepackage{hyperref}
\usepackage{stmaryrd}
\usepackage{enumitem}

\usepackage{tikz}
\usetikzlibrary{positioning}
\usetikzlibrary{arrows}

\usepackage{mathtools}

\usepackage[utf8]{inputenc}
\usepackage{multicol}
\hypersetup{
colorlinks=true, 
breaklinks=true, 
urlcolor= blue, 
linkcolor= blue, 
citecolor={red},
bookmarksopen=true, 
pdftitle={}, 
pdfauthor={GGJJ}, 
pdfsubject={} 
}

\newcommand{\R}{\mathbb{R}}

\newcommand{\N}{\mathbb{N}}
\newcommand{\eqdef}{\stackrel{\text{\tiny{def}}}{=}}   

\newcommand{\Lip}{\mbox{\rm Lip}}

\newcommand{\be} {\begin{eqnarray}}
\newcommand{\ee} {\end{eqnarray}}
\newcommand{\Bea} {\begin{eqnarray*}}
	\newcommand{\Eea} {\end{eqnarray*}}
\newcommand{\pa} {\partial}
\newcommand{\re}{\mathbb{R}}

\newcommand{\al} {\alpha}
\newcommand{\rr}{\rightarrow}

\newcommand{\dip}{\displaystyle}
\newcommand{\B} {\beta}
\newcommand{\de} {\delta}
\newcommand{\g} {\gamma}

\newcommand{\p}  {\prime}
\newcommand{\e}  {\varepsilon}
\newcommand{\De} {\Delta}

\newcommand{\la} {\lambda}
\newcommand{\si} {\sigma}

\newcommand{\f}{\infty}

\newcommand{\noi} {\noindent}

\newcommand{\ga}{\gamma}

\def\abs#1{\vert#1\vert}

\newtheorem{theorem}{Theorem}[section]
\newtheorem{corollary}{Corollary}[section]
\newtheorem{proposition}{Proposition}[section]
\newtheorem{lemma}{Lemma}[section]
\newtheorem{definition}{Definition}[section]
\newtheorem{remark}{Remark}[section]

\title{\bf Optimal regularity for all time 
 for entropy solutions of conservation laws in $BV^s$ }

\author{
Shyam Sundar  Ghoshal\thanks{TIFR, Centre for Apllicable Mathematics,  Bangalore, India. ghoshal@tifrbng.res.in / s.s.ghoshal@gmail.com},
Billel Guelmame\thanks{Universit\'{e} C\^ote d'Azur, CNRS \& Inria, LJAD, 
France. billel.guelmame@unice.fr}, \\
Animesh Jana\thanks{TIFR, Centre for Apllicable Mathematics,  Bangalore, India. animesh@tifrbng.res.in},
St\'ephane Junca\thanks{Universit\'{e} C\^ote d'Azur, CNRS \& Inria, LJAD, France. 
stephane.junca@unice.fr}
}

\date{}

\begin{document}
\everymath{\displaystyle}
\noindent

\bibliographystyle{plain}

\maketitle
\vspace{-.5cm}
\abstract{ 
This paper deals with the optimal regularity for entropy solutions of conservation laws. 
For this purpose, we use two key ingredients: (a) fine structure  of entropy solutions and (b) fractional $BV$ spaces.
We show that optimality  of the regularizing effect for the initial value problem  from $L^\f$ to fractional Sobolev space 
and fractional $BV$ spaces 
is valid for all time.  
Previously, such optimality was proven only for a finite time, before the nonlinear interaction of waves. Here for some well-chosen examples,  the sharp regularity is obtained after the interaction of waves. Moreover, we prove sharp smoothing in $BV^s$ for a convex scalar conservation law with a linear source term.
 Next, we  provide  an upper bound of the maximal smoothing effect for nonlinear scalar multi-dimensional conservation laws and some hyperbolic systems in one or  multi-dimension. \\

}


\tableofcontents

\section{Introduction}
\quad
For nonlinear conservation laws, it is known since Lax-Ole\u{\i}nik \cite{Lax,O57} that the entropy solution can have a better regularity than the initial data for Burgers type fluxes. Such smoothing effect has been obtained in fractional Sobolev spaces \cite{LPT} and recently in fractional $BV$ space \cite{BGJ6} for more general fluxes. The optimality of such regularization is largely open in general. For scalar 1-D conservation laws, there are some optimal results proven up to finite time \cite{CJ1,DW,EM}. The aim of this article is to obtain the same optimality for all time.
 
 
We start with the one-dimensional scalar conservation laws which reads as follow:
\begin{eqnarray}
\frac{\pa u}{\pa t}+\frac{\pa f(u)}{\pa x} &=&0\quad\quad\mbox{for }x\in\R,\,t>0,\label{eqn1}\\
u(x,0)&=&u_0(x)\mbox{ for }x\in\R.\label{eqn2}
\end{eqnarray}
The classical well-posedness theory for the Cauchy problem \eqref{eqn1}--\eqref{eqn2}  is available for $L^\f$ and $BV$ initial data  \cite{Kruzkov,Lax,O57}. $BV$-regularizing effect on entropy solutions has been established in \cite{Lax,O57} for uniformly convex fluxes. It is well know that if the flux function is not uniformly convex then in general, the entropy solution of (\ref{eqn1}) may not have a finite total variation, \cite{AGV,Cheng83}. It can be shown that in one dimension if $f^{\p\p}$ vanishes at some point then there exists a class of initial data such that $f$ can not regularize the corresponding entropy solution up to $BV$ for all time \cite{C2flux}. Hence, to understand the optimal regularity of the entropy solution of (\ref{eqn1}), one works with more general space like  fractional Sobolev space $W^{s,p}$ and  fractional $BV$ spaces $BV^s$, $0<s<1$, $1 \leq p$. 

The advantage of  $BV^s$ spaces  is to recover the fractional Sobolev regularity $W^{\si,p}$ for all $\si<s,1\leq p<s^{-1}$ 
 and to get the $BV$ like trace properties of entropy solutions \cite{COW,P05,P07}. In one dimension, existence of the entropy solutions of (\ref{eqn1}) in $BV^s$,  with $BV^s$ data has been done in \cite{BGJ6} and   with $L^\infty$ data  in the same spaces in \cite{BGJ6, EM, EMproc}. For non-convex fluxes  a Lagrangian framework is used \cite{BM16,BM17}.
  For the scalar 1-D case,  the $BV^s$  smoothing  effect corresponds to the optimal smoothing effect conjectured by Lions, Perthame and Tadmor  in Sobolev spaces with the same fractional derivative $s$ \cite{LPT}. In multi-dimension, for a $C^{2,\g}$ flux, it has been shown \cite{C2flux} that entropy solutions do not need to have fractional derivative $s+\e$ for $\e>0$. For multi-dimensional scalar conservation laws, regularizing effect in fractional Sobolev space was first studied in \cite{LPT}. We refer \cite{TT07} for the best known result in this direction and \cite{Jabin} for further improvement with some extra assumptions, see also \cite{Gess-Lamy} for such results with a source term. The proof of optimality of the exponent $s>0$ is limited  to some one-dimensional scalar examples  \cite{CJ1,DW} and  before the nonlinear interaction of waves. It has been extended for the scalar multi-dimensional case in \cite{CJ2,J4} but not for all time. Recently, in \cite{C2flux} it has been shown that in multi-dimension for any $C^2$ flux $f$ there exists initial data $u_0$ such that the corresponding entropy solution is not in $BV$ for all time $t>0$. 



The present article resolves the following: 
\begin{itemize}
\item 
In one dimension, the optimal  smoothing effect in  fractional  $BV$ spaces is known for the equation (\ref{eqn1}) in bounded strip of time $(0,T)$ for $T>0$, before the interactions of waves \cite{CJ1}. So it is natural to ask the following question:
\begin{equation}\tag{\textbf{Q}{\tiny }}\label{Q1}
\begin{array}{c}
\parbox{13cm}{\emph{Does there exists an entropy solution to \eqref{eqn1} with compact support such that it does not belong to $BV^{s+\e}$ for all $t>0,\,\e>0$?}}
\end{array}
\end{equation}
where $s$ depends on the non-linearity of flux function. We first obtain an entropy solution to \eqref{eqn1} for flux $f(u)=(1+p)^{-1}\abs{u}^{1+p}$ such that $TV^{s+\varepsilon}(u(\cdot,t))=\f$, for all $\varepsilon>0$, for all $t>0$, whereas $TV^{s}(u(\cdot,t))<\f$ with $s=p^{-1}$. Later we generalize this result for a larger class of fluxes.
\item We extend the above result to  higher dimension under some smooth regularity assumption on the flux in section \ref{sec:multi-D}.
\item We are also able to answer the question \eqref{Q1} for entropy solutions to balance laws which read as follow  where $\alpha \in L^\infty((0,+\infty,\R)$:
\begin{eqnarray}
\frac{\pa u}{\pa t}+\frac{\pa f(u)}{\pa x} &=&\al(t)u\quad\quad\mbox{for }x\in\R,\,t>0,\label{balance_laws}\\
u(x,0)&=&u_0(x)\quad\quad\mbox{for }x\in\R.  \label{balance_laws0}
\end{eqnarray}

\item  Smoothing effect for balance laws of type \eqref{balance_laws} in fractional BV space is not known. Based on a recent  Lax-Ole\u{\i}nik type formula \cite{ASG} we prove the $BV^s$ regularizing effect  for entropy solutions to \eqref{balance_laws} with a  convex flux satisfying the $p$-degeneracy {power law} condition \cite{BGJ6},
\begin{equation}\label{flux-condition}
\frac{\abs{f^{\p}(u)-f^{\p}(v)}}{\abs{u-v}^{p}}\geq c_0>0\mbox{ for }u\neq v\in [-M,M].
\end{equation}
  We recall that if $f \in C^2([-M,M])$ then $p \geq 1$ \cite{BGJ6}. 
The exact power-law degneracy is given by the infinimum of $p$ satisfying \eqref{flux-condition}. When $f$ is smooth, the infinimum is a minimum \cite{BGJ6}.

\item Finding an upper bound of the  maximal regularity, even in a given strip ($\re\times[0,T]$),  was unknown for some  triangular systems of conservation laws e.g.,  a pressure swing adsorption system. We answer this question for  a class of 1-D triangular systems and the multi-D Keyfitz-Kranzer system in section \ref{sec:sys} and \ref{sec:KK-sys} respectively. 

\end{itemize}

To provide an answer to the question \eqref{Q1} we recall some of the previously constructed examples \cite{AGG,SC,AGV}. In Section \ref{sec:1D}, Theorem \ref{thm1} provides the direct answer to \eqref{Q1} for  a power-law type flux function  $f(u)=(1+p)^{-1}\abs{u}^{1+p}$. We have discussed before that convex flux function with $p$-degeneracy (i.e., satisfying \eqref{flux-condition}) gives a regularizing  effect in $BV^{s}$ with $s= 1/p$ \cite{BGJ6}. We construct an entropy solution $u$ to \eqref{eqn1} such that $TV^{s+\e}(u(\cdot,t))=\f$ for all $t>0$ and $\e>0$ with $s=1/p$. Following  the constructions in \cite{SC,AGV} we build the entropy solution $u$, consisting infinitely many shock profiles in a compact interval. These shock profiles are named \textit{Asymptotically Single Shock Packet (ASSP)} in \cite{AGG}. Loosely speaking an \textit{ASSP} is a solution with a special structure between two parallel lines in  the half plane $\re_x\times\re_t^+$ such that in large time only one shock curve appears between them. \textit{ASSP} plays a role to describe structure and large time behaviour of the entropy solution to strictly convex flux \cite{AGG}. For the more complex structures of solutions for non-convex fluxes we refer interested reader to \cite{BM16,BM17}. The construction is done in Section \ref{sec:1D}. The building block of such solutions has a support in half strip $[a,b]\times\re_+$ for some $a<b$ and having an oscillation of amount $\de_n$ up to time $t_n$. Then we club all of these building-blocks  to get a solution with the same regularity for all time. 
Similar type of constructions for a slightly different aspect have been used in \cite{CPAM,Ghoshal-jde} to build non $BV$ solutions of scalar conservation laws with discontinuous fluxes.
A larger class of non-uniformly convex fluxes  has been considered in  \cite{AGV}  to build  non-$BV$ solutions for all time. Such a flux $f$ satisfies, 
\begin{equation}\label{decay_condition}
0<f^{\p}(a)-f^{\p}(b)\leq C (a-b)^{q}\mbox{ for all }b<\theta_f<a
\end{equation}
where $q>1$ and $C>0$, which implies that  $f^{\p\p}(\theta_f)=0$. 
Condition \eqref{decay_condition} is about a  minimal degeneracy condition of $f^{\p}$ near the point where $f^{\p\p}$ vanishes. We answer the question \eqref{Q1} for this general class of convex functions satisfying \eqref{decay_condition}.

 The notion of \textit{ASSP} has been generalized recently in \cite{ASG} for balance laws of type \eqref{balance_laws}. Based on a Lax-Ole\u{\i}nik type formula \cite{ASG} for entropy solutions to \eqref{balance_laws}, we are able to answer analogous version of \eqref{Q1} for such a balance law with linear source term. In Section \ref{sec:1D} we provide a construction to show the optimality of the regularizing effect in balance law set up for power-law type flux functions. Like the case for $\al\equiv0$, the constructed solution for balance law is a juxtaposition of infinitely many \textit{ASSP}. Naturally, for such balance laws, the boundaries of \textit{ASSP} are curves instead  straight lines. Moreover, when $\al\equiv 0$, the case of conservation laws is recovered.  
  We choose to answer to the question  \eqref{Q1} in this slightly more general setting.

\par 
In the remainder of the paper, Sections \ref{sec:multi-D}, \ref{sec:sys} and \ref{sec:KK-sys}, the results obtained for the one dimensional scalar case are used to bound the maximal smoothing effect for solutions of  three different problems, namely, scalar multidimensional equations,  a class of triangular systems and  a  multidimensional system.
For  the multidimensional case, planar waves are used as in  \cite{CJ2,C2flux,J4}.
For a class of triangular systems involving a transport equation,   the main problem is to keep  the linear component bounded and not being a $\delta-$ shock while the nonlinear component belongs to the critical BV space. 
For multi-dimensional Keyfitz-Kranzer system \cite{KK}, it has been shown, \cite{Del-kk} that small $TV$ bound of initial data is not enough to get immediately a $BV$ renormalized solution of the  Keyfitz-Kranzer system. In this article we implement his construction to get a similar blow up in all $BV^s$ spaces, $s>0$.


\section{Fractional $BV$ spaces, $BV^s$, $0 < s \leq 1$ }\label{sec:BVs}

\quad In this section, the definition of generalized $BV(\R)$ spaces are recalled \cite{MO}. Then the multi-D case is stated. 
\begin{definition}[$BV^s(\R,\R)$]
Let  $p=1/s$, the $TV^s$ variation also called the
total $p$-variation  of any real function $v$ is:
\begin{equation}
\mathrm{TV}^s\/v\ =\ \sup_{\{x_i\} \in \mathcal{P}}\ \sum_{i=2}^n |v(x_i)-v(x_{i-1})|^p
\end{equation}
where $\mathcal{P}=\{ \{x_1, \cdots ,x_n \},\ x_1<\cdots<x_n , \  2 \leq n \in \N \} $ is the set of subdivisions 
of $\R$. 

The space $\mathrm{BV}^s(\R,\R)$ is the subset of real functions such that,
\begin{equation}
\mathrm{BV}^s(\R)=\{v, \,  \mathrm{TV}^s ( v)<  \infty \}.
\end{equation}
\end{definition}

Notice that $BV^1=BV$ and $BV^s \subset L^\infty$ for all $0 < s \leq 1$. By convention, we set $BV^0=L^\infty$. 
A similar definition can be used to defined $BV^s(I,\R)$ where $I \subset \R$, only considering the subdivisions of $I$. The factional Sobolev space $W^{s,p}$ can be defined as follows:
\begin{definition}
 Let $\Omega\subset\re^N$ be open. Let $s\in(0,1)$ and $p\in[1,\f)$. By $W^{s,p}(\Omega)$ we denote the set of all $u\in L^p(\Omega)$ such that
 \begin{equation}
 \frac{\abs{u(x)-u(y)}}{\abs{x-y}^{s+\frac{N}{p}}}\in L^p(\Omega\times\Omega).
 \end{equation}
\end{definition}
 It is worth mentioning that $BV^s$ does not coincide with fractional Sobolev space, $W^{s,p}$ but it is closely related to $W^{s,p}$ with the critical exponent $p$ for the traces  theory, that is, $s \, p =1$. More precisely, 
 for all $\varepsilon>0, BV^s_{loc} \subset W^{s-\varepsilon,1/s}_{loc}  \subset  W^{s-\varepsilon,1}_{loc} $ \cite{BGJ6}.  All the examples valid for all times in this article present shocks, so are  discontinuous and  therefore never belong to $W^{s+\varepsilon,1/s}, \forall \varepsilon > 0.$ Thus, 
a non $BV^s$ regularity corresponds to a non Sobolev regularity with the same exponent up to any positive $\varepsilon$. The optimality can also be studied in $BV^s$ and corresponds to the similar Sobolev optimaty. 
Notice also that the estimates in fractional $BV$ spaces can be simpler than in fractional Sobolev spaces as in \cite{CJJ}.   It is the reason why only result in $BV^s$ are given in this paper. 
    
Furthermore, $BV^s$ regularity guarantees left and right traces like $BV$ functions. That is why $BV^s$ spaces seem more well fitted to study the regularity of the solutions of conservation laws than the corresponding Sobolev spaces.

To extend the definition of $BV^s$ for the multi-D case, a theorem characterizing $BV^s$ in 1-D is used. 
This theorem characterizes the space   $BV^s$     with the Holder space $\Lip^s$ and  the $BV$ space.  It is due to Michel  Bruneau \cite{Bruneau74}. 
\begin{theorem}[{\textcolor{orange}{Bruneau,  1974}}] \label{thm:Bruneau}
For any $u \in BV^s $ there exists  the following factorization by  a  $s-$Holder function and a  $BV$ function,
 $$u \in BV^s \Leftrightarrow \exists\; L \in \Lip^s(\R,\R), \;  \exists\;  v  \in BV(\R) \mbox{ s.t. } \  u =L \circ v.$$ 
 That means that
 $$BV^s(\R,\R) = \Lip^s (\R,\R)\circ BV (\R,\R).$$
\end{theorem} 
In order to define $BV^s(\R^m)$, we recall the definition of $BV(\R^m)$ for $m\geq 1$
\begin{definition}[ $ BV(\R^m) $]
A function $u$ belongs to $ BV(\R^m) $ if there exists a Radon measure $\mu$ such that 
$$ \int_\mathds{R} u(x)\, \mathrm{div}\, \phi(x)\, \mathrm{d}x\ =\ - <\mu,\phi> \qquad \forall \phi \in \mathcal{C}^1_c(\mathds{R}^m). $$
\end{definition}

Now, the following natural definition of $BV^s(\R^m)$ is proposed for $m\geq 1$.
\begin{definition}[ $ BV^s(\R^m) $]
 A function $u$ belongs to $ BV^s(\R^m) $ if there exists  the following factorization by   an $s-$Holder function $  L \in \Lip^s(\R,\R)$ and a  $BV(\R^m,\R)$ function $v$ such that 
 $$ u =L \circ v.$$ 
\end{definition}
\noindent That means that
 \begin{equation}
 BV^s(\R^m,\R) = \Lip^s (\R,\R)\circ BV (\R^m,\R).
  \end{equation}
This definition can be extended to $BV^s_{loc}(\R^m)$ by: 
 \begin{equation}\label{def:BVs:loc}
 BV^s_{loc}(\R^m,\R) = \Lip^s (\R,\R)\circ BV_{loc} (\R^m,\R).
  \end{equation}   
Notice that the Holder function has to be globally on $\R$  an Holder function since $BV(\R^m)$ is not a subset of $L^\infty$ for $m>1$.

 This definition is valid for $m=1$ thanks to Bruneau's Theorem \ref{thm:Bruneau}.   Moreover, a $BV^s_{loc}(\R)$  1-D  function can be also considered as a $ BV^s_{loc}(\R^m) $  multi-D  function by the following lemma. 
 \begin{lemma} \label{lem:1-multi-D}  Let $\xi \in S^{m-1}$ and $U(X)= u (\xi \cdot X)$,
   $U \in BV^s_{loc}(\R^m)$ if and only if   $u \in BV^s_{loc}(\R)$
 \end{lemma}
 \begin{proof}
    From  the Bruneau's Theorem 
 \ref{thm:Bruneau}, slightly extended on bounded set, $u(x) = L(v(x))$ where $v \in BV_{loc}(\R)$. Let $V(X)$ be $ v (\xi \cdot X)$.  $V$ belongs to $BV_{loc}(\R^m)$ \cite{AFP,Giusti}. Thus  $U(X)= L(V(X))$ belongs to
     $BV^s_{loc}(\R^m,\R) $. The converse is also clear.
 \end{proof}

\section{Sharp regularity  for scalar  1D entropy solutions}  \label{sec:1D}

\quad In this section, we will build some examples to show the optimality of  smoothing effect in $BV^s$  for all time.  This regularity has been obtained in  \cite{BGJ6, GJC,CJJ,EM, EMproc}. The optimality for all time is new. For that purpose, we consider the flux $f(u)=|u|^{p+1}/(p+1)$ so $f'(u)=u|u|^{p-1}$. It is shown  that for $u_0 \in L^\infty$, the solution becomes instantly in $BV^s_{loc}$, with $s=p^{-1}$. Theorem \ref{thm1} stated below shows that the regularizing in $BV^s$ space is optimal for all time since   there exist entropy solutions $u$ such that $ u(\cdot,t) \notin BV^{s+\varepsilon}$ for all $ \varepsilon>0$ and for all $t>0$. The construction of this example is similar to the one done in \cite{AGG} to show infinitely many shock profile occurrence in compact interval. Similar construction has been also used in \cite{AGV} to show the existence of an entropy solution which does not belong to $BV$ for all time. Here we use it to show the existence of an entropy solution which is exactly  in $BV^s$ with $s=p^{-1}$ for all time $t>0$ with no more regularity.

\begin{theorem} \label{thm1}
There exists compactly supported   initial data $u_0\in L^\f(\R)$ such that the corresponding  entropy solution $u(\cdot,t)\in L^\f(\R\times [0,\f))$ of the scalar conservation law  (\ref{eqn1}) with the flux  $f(u)=|u|^{p+1}/(p+1)$, $p\geq 1$ satisfies   $ \mbox{ for all } t>0, \mbox{ for all } \varepsilon>0$  with $s=p^{-1}$, 
$$TV^{s }u(\cdot,t) <  + \f   = TV^{s+\varepsilon}u(\cdot,t).
 $$
\end{theorem}
 Theorem \ref{thm1} can be seen as a particular case (that is, $\al\equiv0$) of the following result stated in context of balance laws.

\begin{theorem}\label{theorem:balance_laws}
	There exists an initial data $u_0\in L^\f(\re)$ such that the corresponding entropy solution to balance law \eqref{balance_laws} with flux $f(u)=(1+p)^{-1}\abs{u}^{1+p}$ for $p>1$ and $\al\in L^\f(0,\f)$ satisfies the following with $s=p^{-1}$
	\begin{equation}\label{eqn:inTheorem2}
	TV^s(u(\cdot,t))<\f=TV^{s+\e}(u(\cdot,t))\mbox{ for all }t>0\mbox{ and }\e>0.
	\end{equation}	
\end{theorem}
{Theorem \ref{theorem:balance_laws} also states about the regularizing of entropy solution corresponding to a particular initial data $u_0$ and flux $f(u)=(p+1)^{-1}\abs{u}^{p+1}$. Next we will show that it is not restricted to a special choice of data and flux. If a flux satisfies a $p$-degeneracy condition like \eqref{flux-condition} then regularizing is valid for any $L^{\f}$ initial data. More precisely, we have the following}
\begin{theorem}\label{theorem:positive}
	Let $f\in C^1(\re)$ be a convex flux satisfying a power-law condition \eqref{flux-condition} and super linear growth condition \eqref{eqn:super_linear}. Let $\al\in L^{\f}(0,\f)$. Let $u_0 \in L^\infty(\mathbb{R})$. Let $u$  be the entropy solution  of the initial value problem for the balance law 
	\eqref{balance_laws},  with the initial data $u_0$  \eqref{balance_laws0}, then 
	\begin{equation}
	u(\cdot,t)\in BV^s_{loc}(\re)\mbox{ for }s=\frac{1}{p}\mbox{ and } \forall \, t>0.
	\end{equation}	
\end{theorem}

As we have discussed before for $\al\equiv0$ case, that for entropy solutions to \eqref{eqn1}, uniformly convex flux regularizes the solution in $BV$ space \cite{Lax,O57} and it fails once we drop the uniform convexity assumption  on flux function \cite{AGV,Cheng83}. As a natural extension, one can ask for the regularizing effect for strictly convex fluxes and it has been shown in \cite{BGJ6} that regularizing is valid in fractional $BV$ space once the flux satisfying a $p$-degeneracy condition  \eqref{flux-condition}. {For strictly convex Lipschitz flux, regularizing effect can be obtained in more general spaces like $BV^{\Phi}$ with a special choice of $\Phi$, \cite{GJC}}.
To prove $TV^{s+\e}(u(\cdot,t))=\f$ for all $t>0$ we construct an entropy solution consisting \textit{ASSP}'s (see \cite{AGG} for more detail on \textit{ASSP}). The other part, that is, $u\in BV^s$ follows from \cite{BGJ6} in the case of Theorem \ref{thm1} that is, when $u$ solves \eqref{eqn1}. But for balance law of type \eqref{balance_laws} no such result exists. It can be proved in a similar fashion as it was done in \cite{BGJ6} for conservation laws. We first give a brief sketch of the proof for $u\in BV^s$ where $u$ is the entropy solution to balance law \eqref{balance_laws}. In order to do that let us first recall some of the definitions and results from \cite{ASG}.
\begin{definition}
	Let $\al\in L^{\f}(0,\f)$. Let $\B$ be primitive of $\al$, that is,
	\begin{equation}\label{defn:beta}
	\B(t)=\int\limits_{0}^{t}\al(\theta)\,d\theta.
	\end{equation}
	Suppose that the flux $f$ is having super-linear growth, that is,
	\begin{equation}\label{eqn:super_linear}
	\lim\limits_{\abs{v}\rr\f}\frac{f(v)}{\abs{v}}=\f.
	\end{equation}
	We define $\Psi:\R\times\R_+\rr\R$-function as follows:
	\begin{equation}\label{defn:Psi}
	x=\int\limits_{0}^{t}f^{\p} \left (\Psi(x,t)e^{\B(\theta)}\right)\,d\theta\,\mbox{ for each }x\in\re.
	\end{equation}
\end{definition} 
Note that $\Psi$ in \eqref{defn:Psi} is well-defined on $\R\times\R_+$ due to super-linear growth \eqref{eqn:super_linear} of $f$ (see \cite{ASG}). For $\al\equiv0$ and strictly convex $C^1$ flux $f$, the $\Psi$-function is nothing but $(f^{\p})^{-1}(x/t)$. As it is observed in \cite{ASG}, the map $x\mapsto \Psi(x,t)$ is increasing for strictly convex flux $f$.
\begin{proposition}(\cite{ASG})\label{prop1}
	Let $\al\in L^\f(\re_+)$ and the flux $f\in C^1(\re)$ satisfying \eqref{eqn:super_linear}. Let $u$ be the entropy solution to \eqref{balance_laws} with initial data $u_0\in L^\f(\re)$. Then  $u$ satisfies
	\begin{equation}
	u(x,t)=e^{\B(t)}\Psi(x-y(x,t),t)\mbox{ for all }x\in\re,t>0
	\end{equation}
	for some function $y$ such that $x\mapsto y(x,t)$ is non-decreasing and $\B$ is defined as in \eqref{defn:beta}. Moreover, for each $T>0$ there exists a constant $C(T)$ such that
	\begin{equation}\label{finite_speed}
	\abs{x-y(x,t)}\leq C(T)t\mbox{ for all }t\in[0,T].
	\end{equation}
	
\end{proposition}

Since $\Psi$ is increasing in its first variable we have the following lemma.
\begin{lemma}\label{lemma:est1}
	Let $f\in C^1$ be a convex flux satisfying the super linear growth condition \eqref{eqn:super_linear} and power-law condition \eqref{flux-condition} with $p\geq1$. Then for any $z_1,z_2\in\re$ we have with $s=p^{-1}$,
	\begin{equation}\label{ineq:lemma1}
	\abs{\Psi(z_1,t)-\Psi(z_2,t)}\leq 
	 \left(\frac{ \abs{z_1-z_2}}{ c_0\gamma(t)}  \right)^{s}, 
	\end{equation}
	with
	\begin{equation}\label{defn:gamma}
	\gamma(t):= \dip\int\limits_{0}^{t}e^{p\B(\theta)}\,d\theta\mbox{ where $\B$ is defined as in \eqref{defn:beta}.}
	\end{equation}
\end{lemma}
\begin{proof}
	Fix two points $z_1,z_2\in\R$. Without loss of generality, we can assume $z_1>z_2$. Since $\Psi$ is increasing in its first variable, we have 
	\begin{equation}
	\Psi(z_1,t)e^{-\B(\theta)}\geq\Psi(z_2,t)e^{-\B(\theta)}.
	\end{equation}
	Since $f^{\p}$ satisfies $p$-degeneracy condition {\eqref{flux-condition} and $f'$ is continuous, that means that $f'$ is monotone.  Assume that $f'$ is increasing, so the absolute values are skipped in \eqref{flux-condition} and},
	\begin{equation}
	f^{\p}(a)-f^{\p}(b)\geq c_0(a-b)^{p}\mbox{ for }a\geq b.
	\end{equation}
	Therefore, by the definition \eqref{defn:Psi} of $\Psi$ we also have for $z_1 > z_2$,
	\begin{eqnarray}
	\abs{z_1-z_2}=z_1-z_2&=&\int\limits_{0}^{t}f^{\p}\left(\Psi(z_1,t)e^{\B(\theta)}\right)\,d\theta-\int\limits_{0}^{t}f^{\p}\left(\Psi(z_2,t)e^{\B(\theta)}\right)\,d\theta\\
	&\geq &\int\limits_{0}^{t}c_0\left(\Psi(z_1,t)-\Psi(z_2,t)\right)^{p}e^{p\B(\theta)}\,d\theta\\
	&=&c_0\abs{\Psi(z_1,t)-\Psi(z_2,t)}^{p}\int\limits_{0}^{t}e^{p\B(\theta)}\,d\theta.
	\end{eqnarray}
	This proves the inequality \eqref{ineq:lemma1}.
\end{proof}
Now we are ready to prove the regularity result for entropy solution to balance laws \eqref{balance_laws}.
\begin{proof}[Proof of Theorem \ref{theorem:positive}]
Fix a partition between $a=x_0<x_1<\cdots<x_m=b$. By Proposition \ref{prop1}, we have
\begin{equation}
\sum\limits_{k=1}^{m}\abs{u(x_k,t)-u(x_{k-1},t)}^p=e^{p\B(t)}\sum\limits_{k=1}^{m}\abs{\Psi(x_k-y(x_k,t),t)-\Psi(x_{k-1}-y(x_{k-1},t),t)}^p.
\end{equation}
By virtue of Lemma \ref{lemma:est1} we have
\begin{equation}
\sum\limits_{k=1}^{m}\abs{u(x_k,t)-u(x_{k-1},t)}^p\leq \dip e^{p\B(t)}(c_0\ga(t))^{-1}\sum\limits_{k=1}^{m}\abs{x_k-y(x_k,t)-x_{k-1}-y(x_{k-1},t)}.
\end{equation}
Since $x\mapsto y(x,t)$ is increasing for each fixed $t$, we have
\begin{eqnarray}
\sum\limits_{k=1}^{m}\abs{u(x_k,t)-u(x_{k-1},t)}^p
&\leq&{e^{p\B(t)}}(c_0\ga(t))^{-1}\left[b-a+y(b,t)-y(a,t)\right]\nonumber\\
&\leq&{e^{p\B(t)}}(c_0\ga(t))^{-1}\left[2(b-a)+2C(T)t\right]
\end{eqnarray}
for all $t\in(0,T)$. The last line follows from the inequality \eqref{finite_speed}.  This completes the proof of Theorem \ref{theorem:positive}.
\end{proof}

Our next aim is to establish the optimality of Theorem \ref{theorem:positive} for all time $t>0$ and for that we restrict our discussion for power-law type fluxes, more precisely, $f(u)=(p+1)^{-1}\abs{u}^{p+1}$ for $p>1$. 
\begin{proof}[Proof of Theorem \ref{theorem:balance_laws}] To set the path for constructing  an entropy solution to \eqref{balance_laws} which does not belong to $BV^s_{loc}(\re)$ for $s>1/p$, we first observe the structure of entropy solution to the following initial data. 

\begin{equation}\label{defn:initial_data_n}
u_0^n(x)=\left\{\begin{array}{lll}
0&\mbox{ for }&x<x_n-\De x_n,\\
\de_n&\mbox{ for }&x_n-\De x_n<x<x_n,\\
-\de_n&\mbox{ for }&x_n<x<x_n+\De x_n,\\
0&\mbox{ for }&x_n+\De x_n<x
\end{array}\right.
\end{equation}
for $\de_n,\,\De x_n>0$ and $x_n\in\re$. For $f(u)=(p+1)^{-1}\abs{u}^{p+1}$, $\Psi$ has the following  form
\begin{equation}\label{Psi:power-law}
\Psi(x,t)= {\textstyle {x}} \abs{{\textstyle {x}}}^{-\frac{p-1}{p}}\ga(t)^{-\frac{1}{p}}
\end{equation}
where $\gamma(t)$ is defined as in \eqref{defn:gamma}. With the help of results from \cite{AGG,ASG} we have the following observations
\begin{enumerate}
	\item Consider a Riemann problem $w_C^0$ defined as follows
	\begin{equation}
	w_C^0=\left\{\begin{array}{rr}
	w_-&\mbox{ for }x<x_0,\\
	w_+&\mbox{ for }x>x_0,
	\end{array}\right.\mbox{ where }w_->w_+.
	\end{equation}
	The entropy solution, $w_C$ to \eqref{balance_laws} corresponding to Riemann data $w_C^0$ has the following form
	\begin{equation}
	w_C=\left\{\begin{array}{rr}
	w_-e^{\B(t)}&\mbox{ for }x<x_0+\lambda (t),\\
	w_+e^{\B(t)}&\mbox{ for }x>x_0+\lambda (t),
	\end{array}\right.
	\end{equation}
	for $t>0$ where $\la(t)$ is defined as follows:
	\begin{equation}
		\la(t):=\frac{1}{w_+-w_-}\int\limits_{0}^{t}\left[f(w_+e^{\B(\theta)})-f(w_-e^{\B(\theta)})\right]e^{-\B(\theta)}\,d\theta.
	\end{equation}
	\item Next we consider a special data $w_L^0$ defined as follows:
	\begin{equation}
	w_L^0(x)=\left\{\begin{array}{lll}
	0&\mbox{ for }&x<x_L,\\
	\de_n&\mbox{ for }&x>x_L,
	\end{array}\right.
	\end{equation}
	where $\de_n>0$. Then entropy solution to \eqref{balance_laws} with initial data $w_L^0$ will look like
	\begin{equation}\label{structure:w_L}
	w_L(x,t)=\left\{\begin{array}{lll}
	0&\mbox{ for }&x<x_L,\\
	\Psi(x-x_L,t)e^{\B(t)}&\mbox{ for }&x_L\leq x\leq\zeta_L(t),\\
	\de_ne^{\B(t)}&\mbox{ for }&x>\zeta_L(t),
	\end{array}\right.
	\end{equation}
    for $t>0$ where $\zeta_L(t)$ are determined by 
    \begin{equation}\label{def:zeta_L}
    \Psi(\zeta_L(t)-x_L,t)=\de_n.
    \end{equation}
	
	\item
	Now consider the following data 
	\begin{equation}
	w_R^0(x)=\left\{\begin{array}{lll}
	-\de_n&\mbox{ for }&x<x_R,\\
	0&\mbox{ for }&x>x_R,
	\end{array}\right.
	\end{equation}
	where $\de_n>0$. Then entropy solution to \eqref{balance_laws} will look like
	\begin{equation}\label{structure:w_R}
	w_R(x,t)=\left\{\begin{array}{lll}
	-\de_ne^{\B(t)}&\mbox{ for }&x<\zeta_R(t),\\
	\Psi(x-x_R,t)e^{\B(t)}&\mbox{ for }&\zeta_R(t)\leq x\leq x_R,\\
	0&\mbox{ for }&x>x_R,
	\end{array}\right.\mbox{ for }t>0,
	\end{equation}
	where $\zeta_R$ is determined by
	\begin{equation}\label{def:zeta_R}
	\Psi(\zeta_R(t)-x_R,t)=-\de_n.
	\end{equation}
	
\end{enumerate}
Let us set $ \bar{x}_0:=x_n$, $x_L:=x_n-\De x_n$ and $x_R:=x_n+\De x_n$. Suppose the corresponding $\zeta_L(t)$ and $\zeta_R(t)$ intersect each other at $(\tilde{x}_n,t_n)$ for the first time. From \eqref{structure:w_L} and \eqref{structure:w_R} we observe that $x_n-\De x_n=x_L\leq \tilde{x}_n\leq x_R=x_n+\De x_n$. By using \eqref{Psi:power-law}, \eqref{def:zeta_L} and \eqref{def:zeta_R} we have
\begin{equation}
(\tilde{x}_n-x_n+\De x_n)^{\frac{1}{p}}\ga(t_n)^{-\frac{1}{p}}=\de_n= (x_n+\De x_n-\tilde{x}_n)^{\frac{1}{p}}\ga(t_n)^{-\frac{1}{p}}.
\end{equation}
Hence, we get $\tilde{x}_n=x_n$. Subsequently, we obtain $(\De x_n)^{\frac{1}{p}}\ga(t_n)^{-\frac{1}{p}}=\de_n$. Recall definition of $\g(t)$ as in \eqref{defn:gamma}. Therefore, $t_n$ is determined as follows,
\begin{equation}\label{eqn:intersection_point}
\int\limits_{0}^{t_n}e^{p\B(\theta)}\,d\theta=\frac{\De x_n}{\de_n^p}.
\end{equation}
Suppose $B^*$ is the integration of $e^{\B}$ over $\re_+$, that is,
\begin{equation}
B^*:=\int\limits_0^\f e^{p\B(\theta)}\,d\theta=\gamma(+\f).
\end{equation}
Note that $B^*\in(0,\f]$. For $\al\equiv0$ case, we have $B^*=\f$. Next consider the case when $B^*<\f$ and
\begin{equation}
\frac{\De x_n}{\de_n^p}\geq B^*.
\end{equation}
In this case, we have the following feature which does not arise for solutions of \eqref{eqn1}:
\begin{equation}
\zeta_L(t)<x_n<\zeta_R(t) \mbox{ for all }t\in(0,\f).
\end{equation}
Summarizing these observations we have
\begin{enumerate}[label=\textbf{O\arabic*}.]
	\item\label{O1} If $\frac{\De x_n}{\de_n^p}< B^*$ then \eqref{eqn:intersection_point} has a unique solution in $(0,\f)$. 
	
	\item\label{O2} If $\frac{\De x_n}{\de_n^p}\geq B^*$ then \eqref{eqn:intersection_point} has no solution in $(0,\f)$. In this case, we set $t_n=\f$, that is to say that $\zeta_L$ and $\zeta_R$ never meet with each other. 
\end{enumerate}

If $t_n<\f$ then note that for $t>t_n$ we have
\begin{equation}
\Psi(x_n-x_L,t)=-\Psi(x_n-x_R,t).
\end{equation}
Therefore we have the following structure of entropy solution $u$, to \eqref{balance_laws} with initial data $u_0$ as in \eqref{defn:initial_data_n}: 
\begin{enumerate}
	\item For $0<t<t_n$ we have
	\begin{equation}
	u^n(x,t)=\left\{\begin{array}{lll}
	0&\mbox{ for }&x<x_n-\De x_n,\\
	\Psi(x-(x_n-\De x_n),t)e^{\B(t)}&\mbox{ for }&x_n-\De x_n<x<\zeta_L(t),\\
	\de_ne^{\B(t)}&\mbox{ for }&\zeta_L(t)<x<x_n,\\
	-\de_ne^{\B(t)}&\mbox{ for }&x_n<x<\zeta_R(t),\\
	\Psi(x-(x_n+\De x_n),t)e^{\B(t)}&\mbox{ for }&\zeta_R(t)<x<x_n+\De x_n,\\
	0&\mbox{ for }&x_n+\De x_n<x.
	\end{array}\right.
	\end{equation}
	
	\item For $t>t_n$ we have
	\begin{equation}
	u^n(x,t)=\left\{\begin{array}{lll}
	0&\mbox{ for }&x<x_n-\De x_n,\\
	\Psi(x-(x_n-\De x_n),t)e^{\B(t)}&\mbox{ for }&x_n-\De x_n<x<x_n,\\
	\Psi(x-(x_n+\De x_n),t)e^{\B(t)}&\mbox{ for }&x_n<x<x_n+\De x_n,\\
	0&\mbox{ for }&x_n+\De x_n<x.
	\end{array}\right.
	\end{equation}
	
\end{enumerate}
Note that $TV^{s+\e}(u^n(\cdot,t))\geq (2\de_n)^{\frac{1}{s+\e}}e^{\frac{\B(t)}{s+\e}}$ for $t\in[0,t_n)$. From the above discussion we know that support of the entropy solution $u_n(\cdot,t)$ lies in $[x_n-\De x_n,x_n+\De x_n]$ for all time $t>0$. We choose $\De x_n=(n\log^2(n+1))^{-1}$ and $\de_n=(n\log^3(n+1))^{-\frac{1}{p}}$. Subsequently, we have
	\begin{equation}\label{convergence1}
	\frac{\De x_n}{\de_n^p}=\log(n+1)\rr\f\mbox{ as }n\rr\f.
	\end{equation}
	Since $\sum\limits_{n=1}^{\f}\De x_n<\f$ we can choose a sequence $x_n$ such that $x_n+\De x_n<x_{n+1}-\De x_{n+1}<x^*<\f$ for all $n\geq1$. Now we define an initial data $u_0$ as follows
	\begin{equation}
	u_0=\sum\limits_{n=1}^{\f}u^n_0.
	\end{equation}
	Note that by previous observation and choice of $x_n$, entropy solutions $u_n$ has mutually disjoint support for all $t>0$. Therefore, the entropy solution $u$ of \eqref{balance_laws} corresponding to initial data $u_0$ can be written as 
	\begin{equation}
	u(x,t)=\sum\limits_{n=1}^{\f}u_n(x,t)\mbox{ for all }x\in\re,\,t>0.
	\end{equation}
	Recall observations \eqref{O1} and \eqref{O2}. Hence, for each fixed $t\in(0,\f)$ there exists an $n_0$ such that $t<t_{n}$ for all $n\geq n_0$ due to \eqref{convergence1}. From definition \eqref{defn:beta} of $\B(t)$ we have $\B(t)\geq-t\|\al\|_{L^\f(\R_+)}$ for all $t\geq0$. Therefore, we have
	\begin{equation}
	TV^{s+\e}(u(\cdot,t))\geq   2^{\frac{1}{s+\e}}e^{\frac{\B(t)}{s+\e}}\sum\limits_{n=n_0}^{\f}\de_n^{\frac{p}{1+p\e}}=2^{\frac{1}{s+\e}}e^{\frac{-t}{s+\e}\|\al\|_{L^{\f}(\R_+)}}\sum\limits_{n=n_0}^{\f}\frac{1}{(n\log^3(n+1))^{\frac{1}{1+p\e}}}=\f.
	\end{equation}
	Note that the $TV^s(u(\cdot,t))<\f$ for $s=1/p$ and $t>0$ due to Theorem \ref{theorem:positive}.
\end{proof} 
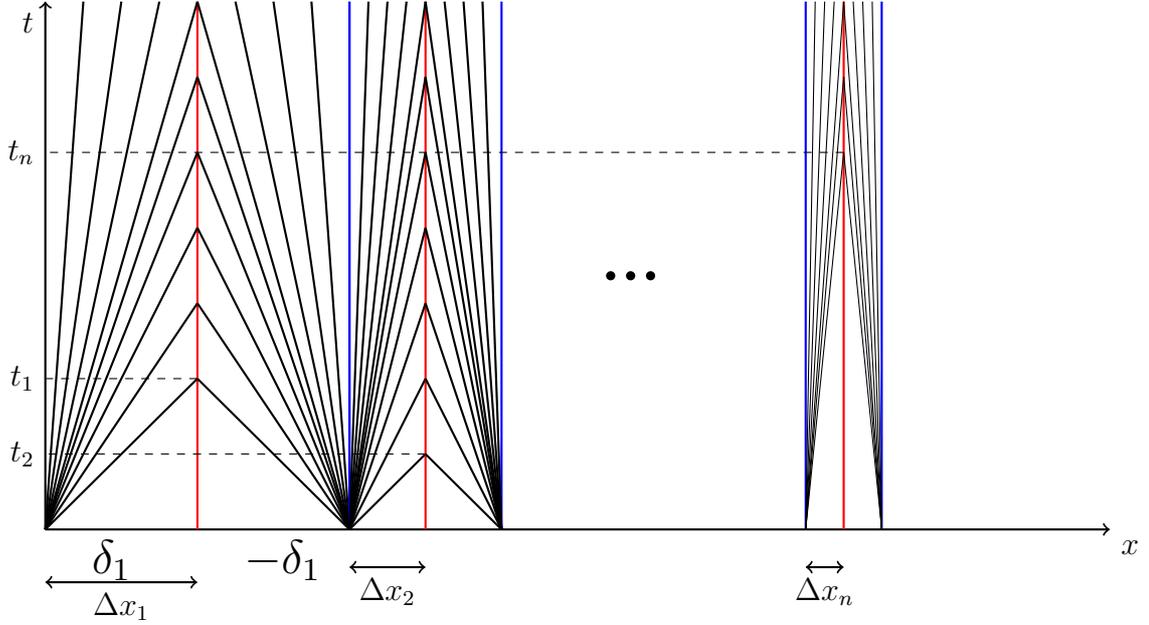
\begin{figure}
	\centering
	\begin{tikzpicture}
	
	\draw[thick,->] (-8 ,0) -- (6,0) node[anchor=north west] {$x$};
	\draw[thick,->] (-8 ,0) -- (-8,7) node[anchor=north east] {$t$};
	
	\draw[thick][color=blue]  (-4 ,0) -- (-4,7) ;
	\draw[thick][color=red]   (-6 ,0) -- (-6,7) ;
	\draw[thick][color=black] (-8 ,0) -- (-6,2) ;
	\draw[thick][color=black] (-8 ,0) -- (-6,3) ;
	\draw[thick][color=black] (-8 ,0) -- (-6,4) ;
	\draw[thick][color=black] (-8 ,0) -- (-6,5) ;
	\draw[thick][color=black] (-8 ,0) -- (-6,6) ;
	\draw[thick][color=black] (-8 ,0) -- (-6,7) ;
	\draw[thick][color=black] (-8 ,0) -- (-6.5,7);
	\draw[thick][color=black] (-8 ,0) -- (-7,7) ;
	\draw[thick][color=black] (-8 ,0) -- (-7.5,7);
	\draw[thick][color=black] (-4 ,0) -- (-6,2) ;
	\draw[thick][color=black] (-4 ,0) -- (-6,3) ;
	\draw[thick][color=black] (-4 ,0) -- (-6,4) ;
	\draw[thick][color=black] (-4 ,0) -- (-6,5) ;
	\draw[thick][color=black] (-4 ,0) -- (-6,6) ;
	\draw[thick][color=black] (-4 ,0) -- (-6,7) ;
	\draw[thick][color=black] (-4 ,0) -- (-5.5,7) ;
	\draw[thick][color=black] (-4 ,0) -- (-5,7) ;
	\draw[thick][color=black] (-4 ,0) -- (-4.5,7) ;

	\draw[thick][color=blue] (-4 ,0) -- (-4,7) ;
	\draw[thick][color=blue] (-2 ,0) -- (-2,7) ;
	\draw[thick][color=red] (-3 ,0) -- (-3,7) ;
	\draw[thick][color=black] (-4 ,0) -- (-3,1) ;
	\draw[thick][color=black] (-4 ,0) -- (-3,2) ;
	\draw[thick][color=black] (-4 ,0) -- (-3,3) ;
	\draw[thick][color=black] (-4 ,0) -- (-3,4) ;
	\draw[thick][color=black] (-4 ,0) -- (-3,5) ;
	\draw[thick][color=black] (-4 ,0) -- (-3,6) ;
	\draw[thick][color=black] (-4 ,0) -- (-3,7) ;
	\draw[thick][color=black] (-4 ,0) -- (-3.25,7) ;
	\draw[thick][color=black] (-4 ,0) -- (-3.5,7) ;
	\draw[thick][color=black] (-4 ,0) -- (-3.75,7) ;
	\draw[thick][color=black] (-2 ,0) -- (-3,1) ;
	\draw[thick][color=black] (-2 ,0) -- (-3,2) ;
	\draw[thick][color=black] (-2 ,0) -- (-3,3) ;
	\draw[thick][color=black] (-2 ,0) -- (-3,4) ;
	\draw[thick][color=black] (-2 ,0) -- (-3,5) ;
	\draw[thick][color=black] (-2 ,0) -- (-3,6) ;
	\draw[thick][color=black] (-2 ,0) -- (-3,7) ;
	\draw[thick][color=black] (-2 ,0) -- (-2.75,7) ;
	\draw[thick][color=black] (-2 ,0) -- (-2.5,7) ;
	\draw[thick][color=black] (-2 ,0) -- (-2.25,7) ;

	\draw[thick][color=blue] (2 ,0) -- (2,7) ;
	\draw[thick][color=blue] (3 ,0) -- (3,7) ;
	\draw[thick][color=red] (2.5 ,0) -- (2.5,7) ;
	\draw[color=black] (2 ,0) -- (2.5,5) ;
	\draw[color=black] (2 ,0) -- (2.5,6) ;
	\draw[color=black] (2 ,0) -- (2.5,7) ;
	\draw[color=black] (3 ,0) -- (2.625,7) ;
	\draw[color=black] (3 ,0) -- (2.75,7) ;
	\draw[color=black] (3 ,0) -- (2.875,7) ;
	\draw[color=black] (3 ,0) -- (2.5,5) ;
	\draw[color=black] (3 ,0) -- (2.5,6) ;
	\draw[color=black] (3 ,0) -- (2.5,7) ;
	\draw[color=black] (2 ,0) -- (2.375,7) ;
	\draw[color=black] (2 ,0) -- (2.25,7) ;
	\draw[color=black] (2 ,0) -- (2.125,7) ;
	
	\draw[dashed] (-6,2) -- (-8,2) node[anchor= east] {$t_1$};
	\draw[dashed] (-3,1) -- (-8,1)  node[anchor= east] {$t_2$}  ;
	\draw[dashed] (2.5,5) -- (-8,5) node[anchor= east] {$t_n$}  ;
	
	\draw[thick][] (-0.5,1.5) node[anchor=south west, transform canvas={scale=2}]{$\textbf{...} $};
	
	\draw[thick][] (-4.75,-0.6) node[anchor=south, transform canvas={scale=1.5}]{$\delta_1$};
	\draw[thick][] (-3.25,-0.6) node[anchor=south, transform canvas={scale=1.5}]{$-\delta_1$};

	\draw[thick,<->] (-8,-0.7) -- (-6,-0.7) ;
	\draw[thick][] (-7,-0.7) node[anchor=north] {$\Delta x_1$};
	\draw[thick,<->] (-4,-0.5) -- (-3,-0.5) ;
	\draw[thick][] (-3.5,-0.5) node[anchor=north] {$\Delta x_2$};
	\draw[thick,<->] (2,-0.5) -- (2.5,-0.5) ;
	\draw[thick][] (2.25,-0.5) node[anchor=north] {$\Delta x_n$};
	];
	\end{tikzpicture}
	\caption{This picture illustrates the entropy solution constructed in Theorem \ref{thm1} for $\al\equiv0$ case. This construction and structure of entropy solution have been previously studied in \cite{AGG,SC,AGV}.} \label{Example1}
\end{figure}

Our next result upgrades Theorem \ref{thm1} for more general class of functions satisfying the following hypothesis:
\begin{enumerate}[label=(H-\arabic{*}), ref=(H-\arabic{*})]
	\item\label{H-1} $f\in C^1(\re)$ is a strictly convex function such that $f(0)=f^{\p}(0)=0$.
	\item\label{H-2} There exist $q>1$, $r>0$  and $C>0$ such that 
	\begin{equation}\label{flux}
	0\leq f^{\p}(a)-f^{\p}(b)\leq C(a-b)^{q}\mbox{ for all }b\in(-r,0)\mbox{ and }a\in(0,r).
	\end{equation}
\end{enumerate}
\begin{remark}
We do not lose generality by assuming that $f(0)=f'(0)=0$, due to the change of variables $x \mapsto x-f'(0)t$ and $\tilde{f}(u) \eqdef f(u)-f(0)-f'(0)u$.
\end{remark}
The class of function satisfying \ref{H-1} and \ref{H-2} was previously considered in \cite{AGV} to show non-BV propagation for all time $t>0$. In this article, we will show the non-$BV^s$ propagation for same class of function in the context of balance laws \eqref{balance_laws}.
\begin{theorem}\label{thm2}
	Let $f\in C^1(\re)$ satisfying \ref{H-1} and \ref{H-2} along with super linear growth condition \eqref{eqn:super_linear}. Let $\al\in L^{\f}(0,\f)$. Then there exists a compact support initial data $u_0\in L^{\f}(\R)$ such that the corresponding entropy solution $u$ of \eqref{balance_laws} satisfies the following:
	\begin{equation}
	u(\cdot,t)\notin BV^s_{loc}(\re)\mbox{ for all }s>1/q,\, t>0.
	\end{equation}
\end{theorem}
\begin{remark}
	Note that Theorem \ref{thm2} is optimal for the class of fluxes satisfying \eqref{flux}. It is easy to verify that $f(u)=(q+1)^{-1}|u|^{q+1}$ satisfies \eqref{flux} and as we have seen in Theorem \ref{theorem:positive}, $u(\cdot,t)\in BV^{1/q}$ for $t>0$.
\end{remark}

 \noi\textbf{Observation:} We want to make a remark that for convex $f$ satisfying \ref{H-1}, we have $\Psi(0,t)=0$. From the definition of $\Psi$ we have
 \begin{equation}\label{est:int0}
 0=\int\limits_{0}^{t}f^{\p}\left(\Psi(0,t)e^{\B(\theta)}\right)\,d\theta.
 \end{equation}
 Since $f$ is a $C^1$ strictly convex function, $f^{\p}$ is increasing. Hence $af^{\p}(a)>0$ for any $a\neq0$ because $f^{\p}(0)=0$. 
 Suppose $\Psi(0,t)>0$ then
 \begin{equation}\label{est:int1}
  \int\limits_{0}^{t}f^{\p}\left(\Psi(0,t)e^{\B(\theta)}\right)\,d\theta>0.
 \end{equation}
 Similarly, if $\Psi(0,t)<0$, then we have
 \begin{equation}\label{est:int2}
 \int\limits_{0}^{t}f^{\p}\left(\Psi(0,t)e^{\B(\theta)}\right)\,d\theta<0.
 \end{equation}
 Note that both \eqref{est:int1} and \eqref{est:int2} contradict with \eqref{est:int0}. Therefore we have $\Psi(0,t)=0$. Note that $\Psi$ is increasing in its first variable due to strict convexity assumption on $f$. Subsequently, we get $x\Psi(x,t)>0$ for any $x\neq0$.
 
 Before we give the main construction to prove Theorem \ref{thm2} we first recall some results from \cite{AGV} and find structure of the entropy solution to the following data
\begin{equation}\label{data1}
u_{A,B}^0(x)
=\left\{\begin{array}{lll}
0&\mbox{ if }&x\notin[A,B],\\
a_{A,B}&\mbox{ if }&x\in[A,\tau],\\
b_{A,B}&\mbox{ if }&x\in(\tau,B].\\
\end{array}\right.
\end{equation}
Next we make a choice for the pair $(a_{A,B},b_{A,B})$ depending on $A,B$. For that purpose we define $G:\R\rr\R$ as  
\begin{equation}\label{def:G}
G(a)=\int\limits_{0}^{t_0}af^{\p}(ae^{\B(\theta)})-f(ae^{\B(\theta)})e^{-\B(\theta)}\,d\theta,
\end{equation}
where $t_0>0$ is fixed. Now we claim that $z\mapsto G(z)$ is increasing for $z>0$ and decreasing for $z<0$. To see this consider $a>a_1>0$, then by Mean Value Theorem we have
\begin{align*}
&ae^{\B(\theta)}f^{\p}(ae^{\B(\theta)})-a_1e^{\B(\theta)}f^{\p}(a_1e^{\B(\theta)})-f(ae^{\B(\theta)})+f(a_1e^{\B(\theta)})\\
=&ae^{\B(\theta)}f^{\p}(ae^{\B(\theta)})-a_1e^{\B(\theta)}f^{\p}(a_1e^{\B(\theta)})-f^{\p}(c^{*})(a-a_1)e^{\B(\theta)},
\end{align*}
for some $c^{*}\in(a_1e^{\B(\theta)},ae^{\B(\theta)})$. Since $f^{\p}$ is increasing, we get
\begin{align*}
&ae^{\B(\theta)}f^{\p}(ae^{\B(\theta)})-a_1e^{\B(\theta)}f^{\p}(a_1e^{\B(\theta)})-f(ae^{\B(\theta)})+f(a_1e^{\B(\theta)})\\
&\geq ae^{\B(\theta)}f^{\p}(ae^{\B(\theta)})-a_1e^{\B(\theta)}f^{\p}(a_1e^{\B(\theta)})-f^{\p}(ae^{\B(\theta)})(a-a_1)e^{\B(\theta)}\\
&=a_1e^{\B(\theta)}\left[f^{\p}(ae^{\B(\theta)})-f^{\p}(a_1e^{\B(\theta)})\right]\\
&\geq0.
\end{align*}
Hence, from \eqref{def:G} we obtain $a\mapsto G(a)$ is increasing for $a>0$. By a similar argument, we get $b\mapsto G(b)$ is decreasing for $b<0$. Therefore, we have that there exists $r_0>0$ such that for a given $a\in (0,r_0)$ there is a $b\in(-r,0)$ satisfying $G(a)=G(b)$. Let us fix $a_0\in(0,r_0)$ and $b_0\in(-r,0)$ such that $G(a_0)=G(b_0)$. Define $F_-:[b_0,0]\rr\R$ and $F_+:[0,a_0]\rr\R$ as follows
\begin{equation}
F_+(a)=\int\limits_{0}^{t_0}f^{\p}(ae^{\B(\theta)})\,d\theta\mbox{ and }F_-(b)=-\int\limits_{0}^{t_0}f^{\p}(be^{\B(\theta)})\,d\theta.
\end{equation}
Since $f^{\p}$ is increasing and $f^{\p}(0)=0$, $F_+$ is increasing and $F_-$ is decreasing. We also have $F_+(0)=F_-(0)=0$. Now we fix $A,B$ such that 
\begin{equation}\label{AB:bound}
B-A\leq \min\{F_+(a_0),F_-(b_0)\}.
\end{equation}  We wish to find $a_{A,B}\in[0,a_0],b_{A,B}\in[b_0,0]$ such that 
\begin{equation}\label{AB}
G(a_{A,B})=G(b_{A,B})\mbox{ and }B-A=F_+(a_{A,B})+F_-(b_{A,B}).
\end{equation}
Since $B-A\leq \min\{F_+(a_0),F_-(b_0)\}$, by Intermediate Value Theorem there exist $\bar{a}\in[0,a_0],\bar{b}\in[b_0,0]$ such that $B-A=F_+(\bar{a})=F_-(\bar{b})$. Define $\la:=\min\{G(\bar{a}),G(\bar{b})\}$. Without loss of generality, suppose $\la=G(\bar{a})$. Then we set $a_1=\bar{a}$ and $b_1=0$. Hence $B-A=F_+(a_1)+F_-(b_1)$. Now we choose $b_2\in[b_0,0]$ such that $G(b_2)=G(a_1)\in[0,\la]$. Note that $b_2\geq \bar{b}$ since $G(x)$ is decreasing for $x<0$ and $G(b_2)\leq\la\leq G(\bar{b})$. Since $F_-$ is decreasing, we get $0\leq F_-(b_2)\leq F_-(\bar{b})=B-A$. Now by Intermediate Value Theorem, we choose $a_2\in[0,\bar{a}]$ such that $B-A=F_-(b_2)+F_+(a_2)$. Having defined $\{a_k\}_{1\leq k\leq n}\subset[0,\bar{a}]$ and $\{b_k\}_{1\leq k\leq n}\subset[\bar{b},0]$ such that $B-A=F_-(b_n)+F_+(a_n)$ we choose $b_{n+1}\in[\bar{b},0]$ such that $G(b_{n+1})=G(a_{n})$. Note that the choice of $b_{n+1}$ is guaranteed as $0\leq G(a_{n})\leq\la\leq G(\bar{b})$. 
 Subsequently, we get $0\leq F_-(b_{n+1})\leq B-A=F_+(\bar{a})$. Now we choose $a_{n+1}\in[0,\bar{a}]$ such that $B-A=F_-(b_{n+1})+F_+(a_{n+1})$. Hence, by this inductive process we get $\{a_{n}\}_{n\in\N}\subset[0,\bar{a}]$ and $\{b_{n}\}_{n\in\N}\subset[\bar{b},0]$. Since both sequences are bounded, there is a subsequence $n_k$ such that $b_{n_k}\rr b_{A,B}$ and $a_{n_k}\rr a_{A,B}$ as $k\rr\f$ for some $a_{A,B}\in[0,\bar{a}]$ and $b_{A,B}\in[\bar{b},0]$. Since $F_\pm,G$ are continuous functions, by passing to the limit we show that $a_{A,B}\in[0,\bar{a}]$ and $b_{A,B}\in[\bar{b},0]$ satisfy \eqref{AB}.

Suppose $a_{A,B},b_{A,B}$ are satisfying \eqref{AB}. Now we can choose $\tau$ as follows:
\begin{equation}\label{t0}
 \tau+\int\limits_0^{t_0}\frac{f(a_{A,B}e^{\B(\theta)})-f(b_{A,B}e^{\B(\theta)})}{a_{A,B}-b_{A,B}}e^{-\B(\theta)}\,d\theta=A+\int\limits_{0}^{t_0}f^{\p}(a_{A,B}e^{\B(\theta)})\,d\theta.
\end{equation}
Since $G(a_{A,B})=G(b_{A,B})$ and $B-A=F_-(b_{A,B})+F_+(a_{A,B})$ we get
\begin{equation}\label{int0}
a_{A,B}(\tau-A)+b_{A,B}(B-\tau)=0.
\end{equation}

Suppose $u_{A,B}(x,t)$ is the entropy solution to \eqref{balance_laws} for initial data \eqref{data1}. Then as it has been discussed in \cite[section 3]{AGV}, $u_{A,B}$ enjoys the following structure up to time $t_0$:
	\begin{eqnarray}
	u_{A,B}(x,t)&=&\left\{\begin{array}{lll}\label{structure1}
	0&\mbox{ if }&x\notin[\xi_-(t),\xi_+(t)],\\
	\Psi(x-A,t)e^{\B(t)}&\mbox{ if }&\xi_-(t)\leq x\leq \zeta_-(t),\\
	a_{A,B}e^{\B(t)}&\mbox{ if }&\zeta_-(t)<x<\zeta_0(t),\\
	b_{A,B}e^{\B(t)}&\mbox{ if }&\zeta_0(t)<x<\zeta_+(t),\\
	\Psi(x-B,t)e^{\B(t)}&\mbox{ if }& \zeta_+(t)\leq x\leq \xi_+(t),
	\end{array}\right.
	\end{eqnarray}
	where the curves $\xi_\pm,\zeta_\pm,\zeta_0$ are determined as follows
	\begin{align}
	&\Psi(\xi_-(t)-A,t)=0=\Psi(\xi_+(t)-B,t),\\
	&\Psi(\zeta_-(t)-A,t)=a\mbox{ and }\Psi(\zeta_+(t)-B,t)=b,\\
	&\zeta_0(t)=\frac{1}{a_{A,B}-b_{A,B}}\int\limits_{0}^{t}\left(f(a_{A,B}e^{\B(\theta)})-f(b_{A,B}e^{\B(\theta)})\right)e^{-\B(\theta)}\,d\theta.
	\end{align}
	Note that by hypothesis \ref{H-1} $\xi_-(t)=A$ and $\xi_+(t)=B$. By \eqref{t0} we have that two curves $\zeta_\pm$ meet with each other at point $t_0$. For $t\in(t_0,t_0+\De t)$ for small $\De t>0$, the entropy solution $u_{A,B}$ satisfies the following structure
	\begin{eqnarray}\label{structure2}
	u_{A,B}(x,t)&=&\left\{\begin{array}{lll}
	0&\mbox{ if }&x\notin[A,B],\\
	\Psi(x-A)e^{\B(t)}&\mbox{ if }&A\leq x< \zeta_M(t),\\
	\Psi(x-B)e^{\B(t)}&\mbox{ if }& \zeta_M(t)<x\leq B,
	\end{array}\right.
	\end{eqnarray}
	where $\zeta_M(t)$ is the characteristic curve starting at the point $(\tau,0)$. Next we claim that $\zeta_M(t)\in(A,B)$ and  the structure \eqref{structure2} continues to hold for all $t>t_0$. We can prove this in the same way as it was done for \cite[Lemma 3.12]{AGV}. Suppose the curve $t\mapsto\zeta_M(t)$ intersects either $x=A$ line or $x=B$ line. Without loss of generality we assume that $\zeta_M(t)$ first meets $x=B$ line. Therefore there exists a time $t_1>0$ such that at $t=t_1$ we have $\zeta_M(t_1)=B$ and $A<\zeta_M(t)<B$ for $0\leq t<t_1$. Therefore, \eqref{structure2} is valid up to time $t_1$. 
	Consider $\g_\pm$ defined as follows
	\begin{equation}
	\Psi(\g_-(t)-A,t)=\Psi(\zeta_M(t_1)-A,t_1)\mbox{ and }\g_+=\xi_+.
	\end{equation} Since \eqref{structure2} is valid up to time $t_1$, $\g_\pm(t)$ are minimizing curve of the following value function
	\begin{equation}
	U(x,t)=\min\left\{\int_{0}^{\ga(0)}u^0_{A,B}(y)\,dy+\int\limits_0^te^{-\B(\theta)}f^{*}\left(\dot{\g}(\theta)\right)\,d\theta;\,\g:[0,t]\rr\re,\,\g(t)=x\right\}
	\end{equation}
	where $f^{*}$ is the Legendre transform of $f$. By \eqref{t0} we have
	\begin{equation}
	\int\limits_{A}^{B}u^0_{A,B}=0.
	\end{equation}
	Therefore we have
	\begin{equation}
	\int\limits_0^{t_1}e^{-\B(\theta)}f^{*}\left(\dot{\g}_+(\theta)\right)\,d\theta=\int\limits_0^{t_1}e^{-\B(\theta)}f^{*}\left(\dot{\g}_-(\theta)\right)\,d\theta.
	\end{equation}
	By using the definition of $\Psi$ we have
	\begin{equation}\label{eqn:gamma_-}
	\gamma_-(t)-A=\int\limits_{0}^{t}f^{\p}\left(\Psi(\gamma_-(t)-A,t)e^{\B(\theta)}\right)\,d\theta=\int\limits_{0}^{t}f^{\p}\left(\Psi(\zeta_M(t_1)-A,t_1)e^{\B(\theta)}\right)\,d\theta.
	\end{equation}
	Differentiating \eqref{eqn:gamma_-} with respect to $t$, we obtain
	\begin{equation}
	\dot{\ga}_-(\theta)=f^{\p}\left(\Psi(\zeta_M(t_1)-A,t_1)e^{\B(\theta)}\right)\mbox{ for }\theta\in(0,t_1).
	\end{equation}
	Similarly, we have
	\begin{equation}
	\dot{\g}_+(\theta)=0\mbox{ for }\theta\in(0,t_1).
	\end{equation}
	Since $f^{*}\geq0$ we have $f^*(\dot{\g}_-(\theta))=0$ for a.e. $\theta\in[0,t_1]$. Since $0$ is unique minima of $f^*$, we have $\dot{\ga}_-(\theta)=0$ a.e. $\theta\in[0,t_1]$. This gives a contradiction. Hence our claim is proved i.e. $\zeta_M(t)\in(A,B)$ for all $t\geq0$. Now we are ready to prove Theorem \ref{thm2}.

\begin{proof}[Proof of Theorem \ref{thm2}:] 
	Define $A_n,B_n$ as follows:
	\begin{align}
	&A_n=x_n-\frac{1}{n(\log(n+1))^2}\mbox{ and }B_n=x_n+\frac{1}{n(\log(n+1))^2},\nonumber\\
	\mbox{ where }&x_n=4\sum\limits_{k=1}^{n-1}\frac{1}{k(\log(k+1))^2}+\frac{2}{n(\log(n+1))^2}.
	\end{align}
	From the choice of $A_n,B_n$ and $x_n$ it is clear that
	\begin{equation}\label{x0limit}
	B_{n-1}<A_{n}<B_n<A_{n+1}\mbox{ and }\lim\limits_{n\rr\f}x_n= x_0<\f.
	\end{equation}
	Note that $B_n-A_n=2n^{-1}(\log(n+1))^{-2}\rr0$ as $n\rr\f$. Therefore, there exists $n_0\in\N$ such that $A_n,B_n$ satisfy \eqref{AB:bound} for all $n\geq n_0$. By the previous observation, we find $a_{A_n,B_n},b_{A_n,B_n}$ satisfying \eqref{AB} for $B_n-A_n$. Next we define initial data $u_0$ as follows:
	\begin{equation}
	u_0(x)=\left\{\begin{array}{lll}
	u_{A_n,B_n}^0&\mbox{ if }x\in[A_n,B_n]\mbox{ for }n\geq n_0,&\\
	0&\mbox{ otherwise}&
	\end{array}\right.
	\end{equation}
	where $u^0_{A_n,B_n}$ is defined in \eqref{data1}. To simplify the notation we denote $a_n=a_{A_n,B_n}$ and $b_n=b_{A_n,B_n}$. By using \ref{H-2} in \eqref{AB} we get 
	\begin{equation}
	B_n-A_n\leq C(a_n-b_n)^q\int\limits_{0}^{t_0}e^{q\B(\theta)}\,d\theta.
	\end{equation}
	Therefore, we get
	\begin{equation}\label{abAB}
	a_n-b_n\geq c_0^{-\frac{1}{q}}(B_n-A_n)^{\frac{1}{q}}\mbox{ where }c_0=C\int\limits_{0}^{t_0}e^{q\B(\theta)}\,d\theta.
	\end{equation} From \eqref{x0limit} it is clear that $u_0$ has compact support in $\re$. By structure \eqref{structure1} and \eqref{structure2} we know that if $u_{A_n,B_n}$ is the entropy solution to \eqref{balance_laws} for initial data $u^0_{A_n,B_n}$ then the support of $u_{A_n,B_n}$ lies in the strip $[A_n,B_n]\times[0,\f)$. Therefore, if $u(x,t)$ is the solution to \eqref{balance_laws} then $u$ has the following structure
	\begin{equation}\label{structure3}
	u(x,t)=\left\{\begin{array}{lll}
	u_{A_n,B_n}(x,t)&\mbox{ if }x\in[A_n,B_n]\mbox{ for }n\geq n_0,&\\
	0&\mbox{ otherwise.}&
	\end{array}\right.
	\end{equation}
	By $\zeta_n$ we denote the curve $\zeta_M$ appeared in the structure \eqref{structure2} of $u_{A,B}$ corresponding to $A=A_n,B=B_n$.
	From \eqref{structure1} and \eqref{structure2} we obtain the following estimate for any $t>0$,
	\begin{equation}\label{TVs:calc1}
	TV^{s}(u_{A_n,B_n}(\cdot,t))[A_n,B_n]\geq\min\left\{(b_n-a_n)^{\frac{1}{s}}e^{\frac{\B(t)}{s}},\abs{\Psi(\zeta_n(t)-A_n,t)-\Psi(\zeta_n(t)-B_n,t)}^{\frac{1}{s}}e^{\frac{\B(t)}{s}}\right\}.
	\end{equation}
	From the definition of $\Psi$ we have
	\begin{equation}
	\zeta_n(t)-A_n-\zeta_n(t)+B_n=\int\limits_{0}^{t}\left(f^{\p}\left(\Psi(\zeta_n(t)-A_n,t)e^{\B(\theta)}\right)-f^{\p}\left(\Psi(\zeta_n(t)-B_n,t)e^{\B(\theta)}\right)\right)\,d\theta.
	\end{equation}
	 Note that $\zeta_n(t)-A_n>0>\zeta_n(t)-B_n$. Since $\Psi$ is increasing in its first variable and $\Psi(0,t)=0$, we have
	\begin{equation}
	\Psi(\zeta_n(t)-A_n,t)>0>\Psi(\zeta_n(t)-B_n,t).
	\end{equation}
	From the decay condition \eqref{flux} we have
	\begin{equation}
	B_n-A_n\leq C\int\limits_{0}^{t}\left(\Psi(\zeta_n(t)-A_n,t)-\Psi(\zeta_n(t)-B_n,t)\right)^{q}e^{q\B(\theta)}\,d\theta.
	\end{equation}
	Therefore we have
	\begin{equation}\label{estimate:Psi}
	\abs{\Psi(\zeta_n(t)-A_n,t)-\Psi(\zeta_n(t)-B_n,t)}\geq (B_n-A_n)^{\frac{1}{q}}\varrho(t)^{-\frac{1}{q}}\mbox{ where }	\varrho(t):=C\int\limits_{0}^{t}e^{q\B(\theta)}\,d\theta.
	\end{equation}
	Combining \eqref{abAB}, \eqref{TVs:calc1} and \eqref{estimate:Psi} we have
	\begin{equation}\label{TVs:calc2}
	TV^{s}(u_{A_n,B_n}(\cdot,t))[A_n,B_n]\geq\min\left\{c_0^{-\frac{1}{qs}}e^{\frac{1}{s}\B(t)}(B_n-A_n)^{\frac{1}{qs}},(B_n-A_n)^{\frac{1}{qs}}\varrho(t)^{-\frac{1}{qs}}e^{\frac{1}{s}\B(t)}\right\}.
	\end{equation}
	Fix an $s>q^{-1}$. Then there exists $\de>0$ such that $s=q^{-1}+\de$. By our choice of $A_n$ and $B_n$ we have
	\begin{equation}
	B_n-A_n=\frac{2}{n(\log(n+1))^2}.
	\end{equation}
	Since $s=(1/q)+\de$ we have $sq=1+q\de$. We observe that $\B(t)\geq -t\|\al\|_{L^{\f}(\R_+)}$ and $\varrho(t)\leq tCe^{q\|\al\|_{L^{\f}(\R_+)}}$. From definition of $c_0$ we have $c_0\leq t_0Ce^{q\|\al\|_{L^{\f}(\R_+)}}$. Hence, we obtain
	\begin{align}
	TV^{s}(u_{A_n,B_n}(\cdot,t))[A_n,B_n]&\geq 2^{\frac{1}{1+q\de}}e^{\frac{1}{s}\B(t)}\min\left\{c_0^{-\frac{1}{qs}},\varrho(t)^{-\frac{1}{qs}}\right\}n^{-\frac{1}{1+q\de}}(\log(n+1))^{-\frac{2}{1+q\de}}\nonumber\\
	&\geq 2^{\frac{1}{1+q\de}}e^{\frac{-(t+1)}{s}\|\al\|_{L^{\f}(\R_+)}}\min\left\{(Ct_0)^{-\frac{1}{qs}},(Ct)^{-\frac{1}{qs}}\right\}\frac{n^{-\frac{1}{1+q\de}}}{(\log(n+1))^{\frac{2}{1+q\de}}}.
	\label{TVs:calc3}
	\end{align}
	Since $[A_n,B_n],n\geq n_0$ are disjoint intervals we have
	\begin{align*}
	TV^s(u(\cdot,t))&\geq\sum\limits_{n\geq n_0}TV^s(u_{A_n,B_n}(\cdot,t))[A_n,B_n]\\
	&\geq
	2^{\frac{1}{1+q\de}}e^{\frac{-(t+1)}{s}\|\al\|_{L^{\f}(\R_+)}}\min\left\{(Ct_0)^{-\frac{1}{qs}},(Ct)^{-\frac{1}{qs}}\right\}
	\sum\limits_{n\geq n_0}\frac{n^{-\frac{1}{1+q\de}}}{(\log(n+1))^{\frac{2}{1+q\de}}}=\f. 
	\end{align*}
	\end{proof}
\section{The scalar  multi-D case}\label{sec:multi-D}
\quad In this section we deal with $C^\f$--flux function for multi-D scalar conservation laws which reads as follows
\begin{equation}\label{eq:multi-D}
\partial_t  U + \mbox{div}_X F(U)= 0, \qquad U(X,0)=U_0(X).
\end{equation} 
 Non-linearity of a multi-D smooth flux is defined as 
\begin{definition}[Nonlinear flux, \cite{J4}]
Let $F$ belong to  $ C^{\f}([a,b],\R^m)$ $[a,b]$ and  for each $U\in[a,b]$, 
\begin{equation} \label{def:dF}
d_F:=\sup_{ U \in [a,b]}\inf\{k\in\N;\,k\geq1,\mbox{span}(F^{\p\p}(U),\cdots,F^{k+1}(U))=\R^m\} \in \N \cup \{+ \infty\},
\end{equation} 
If $d_F < + \infty$ then $F$ is called a nonlinear flux. \\
 If $d_F=m$ it is called a genuinely nonlinear flux.
\end{definition}
It can be checked \cite{J4} that since $[a,b]$ is compact $d_F[\cdot]$ attains its maximum at some point $\bar{U}\in[a,b]$, so, $d_F$ is well defined in  $ \N \cup \{+ \infty\}$. 
By definition of $d_F$, at least $m$ derivatives of $F'$ are needed to span the $m$-dimensional space $\R^m$ so $d_F \geq m$.
 If $F$ is  a linear flux, $d_F=+\infty$. Notice that for  some exponentially flat fluxes, it is possible  to have $d_F= + \infty$  already in dimension one \cite{CJLO,CJ2}. In this case no $BV^s$ smoothing effect is expected. Indeed, there is a low  smoothing effect in a generalized $BV$ space $BV^\Phi$  \cite{CJLO,EM,EMproc}.  

 For  a $C^\infty$ nonlinear flux $F$,  the Lions, Perthame and Tadmor conjecture  \cite{LPT}
 can be reformulated as follows \cite{J4},
 any entropy solution of \eqref{eq:multi-D} such that $U_0(\R^m) \subset [a,b]$ are regularized in $W^{s,1}$ for all $s<d_F^{-1}$ where $d_F$ is the non-linearity index as in \eqref{def:dF}. 
  The  Lions, Perthame and Tadmor conjecture is still an open problem.
 
 We prove  the limitation of  the regularizing effect  for the class of $C^\infty$ nonlinear  fluxes  $F$ such that $d_F$ is odd.   The restriction of $d_F$ for odd numbers is due to our previous explicit construction in one dimensional case of solution with the exact maximal regularity for all time only for convex fluxes.   
 The existence of an entropy solution  with the conjectured maximal regularity and not more is provided by a construction of a planar wave.  This regularity is not improved for large time.  
  


For a bounded strip of time the limitation of the smoothing effect  for  entropy solutions of multi-dimensional scalar conservation laws  in Sobolev spaces has been already proven in \cite{CJ2,J4}.  On one hand, the limitation for bounded time was due to the difficulty to study in general behaviour of the solutions after interactions of waves as in \cite{Cheng83,CC,DW}. On the other hand, multidimensional fractional $BV$ spaces were not known at that time. Recently, in \cite{C2flux}  it has been shown that given a $C^2$ flux there exists an entropy solution in multi-D such that it is not in $BV_{loc}$ for all time. Authors also prove that there exists an entropy solution which is not in $W^{s+\e,1}_{loc},\forall\e>0$ for all time with $C^{2,\g}$ with  $d_F= 1/s$.

 The point in this section is to obtain the optimality  for all time and in the multi-D $BV^s$ framework. 
To get the optimality for the multi-D case, a planar wave is used as in  \cite{CJ2,C2flux}. 

The flux being nonlinear \cite{LPT} there exist a constant state $\underline{U}$ and a direction $\xi$ such that the flux reachs its  degeneracy  $d_F$ near  $\underline{U}$ and following the direction $\xi$ \cite{J4}. That simply means that the scalar flux $f(u)=  \xi \cdot F(\underline{U}+u)$ has an exact $p$-degeneracy  \eqref{flux-condition} with the optimal $p=d_F$. Moreover, for smooth flux, $p$ is an integer \cite{J4} bigger than the space dimension: 
$ m  \leq p \in \N$. That means that for small $u$ the derivative of the flux $f'$ has exactly a power-law behaviour like $u^p$. For $p$ odd, $f$ is locally convex (or concave) and the Theorem \ref{thm2} can be used. The result reads as follow. 
\begin{corollary}  Let $F$ be a $C^\infty(\R,\R^m)$ flux with  an odd exact degeneracy $d_F=p$ on $[-M,M]$ for some $M>0$  then there exists an entropy solution $U$ of \eqref{eq:multi-D}  such that  $\forall \varepsilon >0, \forall t>0,\,   U(\cdot,t) \in  BV^{s}_{loc}(\R^m,\R)$
and  $U(\cdot,t) \notin  BV^{s+\varepsilon}_{loc}(\R^m,\R)$
 where $s=1/p$.
\end{corollary}
\begin{remark} The parity restriction of $ p $ should be neglected with an implicit and more complicated construction used in \cite{C2flux}. Such a solution does not have the same compact support forever.
\end{remark}

We just recall the main features of the proof in \cite{CJ2} for  the optimality of the $BV^s$ regularity for a bounded time and then  using example of the section \ref{sec:1D} and Lemma \ref{lem:1-multi-D} the optimality for all time follows. 
First, take an example given in the proof of  the Theorem \ref{thm1} with $u_0(x)$ and $u$ the corresponding entropy solution for the flux $f$  and $U_0(X)=\underline{U}+ u_0(\xi \cdot  X)$ then  for all time \cite{CJ2},
$$
U(X,t)=\underline{U}+ u(\xi \cdot  X,t).
$$
The $BV^s$ multi-D regularity of the entropy solution $U$ is the consequence of  1-D optimality of $u$ and Lemma \ref{lem:1-multi-D}.

\section{A class of $2\times 2$ triangular systems}\label{sec:sys}

\quad Getting optimal $BV^s$ solutions for general systems for all time is an open problem. Also, the existence of $BV^s$ solutions for systems is in general open. There are some exceptions, $BV^s$ solutions exist for a gas-chromatography system \cite{BGJP}, a nonlinear acoustics model  and also for diagonal systems \cite{JuLo4}. However, the optimality of the regularity is not yet proven.
In this section, we consider the first example. 
 The  gas-chromatography system is not a Temple system as the well-known chromatography system presented in Bressan's book \cite{Bressan} for instance.  Otherwise, this gas-chromatography system enjoys a nice property in Lagrangian variables \cite{BGJP}, it has the triangular structure:  
\begin{eqnarray}
 \partial_t  u   + \partial_x f(u)  &=& 0\quad\quad\quad\quad\mbox{ for }x\in\re,\,t>0,         \label{eq:1scl}\\
 \partial_t  v   + \partial_x (g(u) v) &=& 0\quad\quad\quad\quad\mbox{ for }x\in\re,\,t>0, \label{eq:lin}\\
 (u,v)(x,0)&=&(u_0,v_0)(x)\mbox{ for }x\in\re.
\end{eqnarray}
%
%
At first sight, the system \eqref{eq:1scl}--\eqref{eq:lin} seems easy to solve. First, one takes the entropy solution of the equation  \eqref{eq:1scl}.  Second, solve the linear transport equation with \eqref{eq:lin}.
But, the velocity of the transport equation is $g(u)$ which can be discontinuous.    For such equations, a Dirac mass can appear \cite{BJ98}.  Thus, due to the  transport equation, such systems are not easy to solve in general.  The pressureless-gas dynamics system is  an example of such problematic systems \cite{BJ99}.

 In this section  we propose optimal $BV^s$ solutions  for two cases. First, a self contained construction for a finite time $[0,T_0]$, $T_0 >0$  where the component $v$ stays continuous. Second, using a recent result of global existence of bounded entropy solutions, we get, as a corollary, the optimality in $BV^s$ for all time. 
 
In the next theorem we construct a solution of the system \eqref{eq:1scl}--\eqref{eq:lin} such that $(u,v)\notin BV^{s+\varepsilon} (\re\times[0,T_0])$ for all $\varepsilon > 0$  and  for power-law type  functions $f$ and $g$ satisfying the following relation, 
\begin{equation}\label{defn_g} g =  h \circ f' ,
\end{equation}
where $h$ is a Lipschitz function.
 We first build a continuous solution $u$ to \eqref{eq:1scl} and then solve \eqref{eq:lin} by using $u$. Similar line of thought has been previously instrumentalized in \cite{Boris} to characterize the attainable set for triangular systems.
\begin{theorem}\label{theorem:sys}
	Let $T>0$. Let $f(u)=|u|^{p+1}/(p+1), p\geq 1$,  $s =1/p$ and $g=h\circ f^{\p}$ where $h$ is a Lipschitz function. Then there exist  compactly supported initial data $u_0, v_0 \in L^\f(\R)$ such that the corresponding entropy solution $(u,v)$ of the triangular system \eqref{eq:1scl}--\eqref{eq:lin} satisfies $\forall  t \in [0,T]$, 
	\begin{equation*}
	u, v \in L^\f([0,T] \times \R),\, TV^{s+\varepsilon}u(\cdot,t)=\f \mbox{ for all }\e>0 \mbox{ and } TV^{s'}v(\cdot,t)=\f\mbox{ for all }s'\in(0,1].
	\end{equation*}
\end{theorem}
\begin{proof}
If $f^{\p}(u(x,t))$ is Lipschitz in $x$--variable then we have $g(u(x,t))$ is Lipschitz in $x$--variable by the choice of $g$. The construction is done in two steps
	
	\noindent\textbf{Step 1:} {\it Construction of a continuous solution of (\ref{eqn1}) such that $f^\prime(u(x,t))$ is Lipschitz in the $x$ variable: }
	\par 	Let $\Delta x_n:=1/\left( n\ \log^2(n+1) \right)$, $t_n:=\frac{\log(n+1)}{\log2} (T+1)>T$ and let $\delta_n:=\left(\frac{\Delta x_n}{t_n}\right)^{\frac{1}{p}}$. We define $x_1=0$ and $x_n=2\sum\limits_{m=1}^n \Delta x_m$ for $n\geq2$. Next we  consider the following  initial data
\[
w_0^n(x)=
   \begin{cases}
     \left(\frac{\Delta x_n-x}{t_n} \right)^s &\mbox{ if } 0 < x \leqslant \Delta x_n, 
\\
    - \left(\frac{x-\Delta x_n}{t_n} \right)^s &\mbox{ if } \Delta x_n \leqslant x < 
2\, \Delta x_n, \\
     0 & \text{ otherwise}.
   \end{cases}
\]
Therefore $\forall x\in\R, \forall\ t>0$,  the  entropy solution $w_n$ of (\ref{eq:1scl}) with the flux $f(u)=|u|^{p+1}/(p+1)$, is given by 
\[
w_n(x,t)=
   \begin{cases}
  \left( \frac{x}{t} \right)^s &\mbox{ if }  0 < x < \min \{ \de_n^p\, t,\,  \Delta x_n 
\}, \\
     \left(\frac{\Delta x_n-x}{t_n-t} \right)^s &\mbox{ if } \de_n^p\, t < x < 
\Delta x_n,  \\
     -\left(\frac{x-\Delta x_n}{t_n-t} \right)^s & \mbox{ if }\Delta x_n < x < 2\, 
\Delta x_n- \de_n^p\, t, \\
        -  \left( \frac{2\, \Delta x_n-x}{t} \right)^s   &\mbox{ if } \max \{ 2\, 
\Delta x_n- \de_n^p\, t,\, \Delta x_n \} < x < 2\, \delta_n, \\
     0 & \text{ otherwise}.
   \end{cases}
\]
Let $u_0(x)= \sum_n w_0^n \left(x-x_n \right)$ and let $u$ be the entropy solution of (\ref{eq:1scl}) with the same flux $f$. Then one can show that $u(x,t)$ is continuous function on $\R\times (0,T]$ and $f^\p(u(x,t))$ is Lipschitz in the $x$ variable. We also have 
$$
TV^{s+\varepsilon}u(\cdot,t)\{[0,2\, \Delta x_n]\} \geqslant 4\, \left( \Delta x_n/t_n \right)^\frac{1}{1+p\varepsilon}\mbox{ for }  t_n > t,\, \varepsilon>0.$$ 		
\noindent\textbf{Step 2:}		 We devote this step to find the component $v$ as in (\ref{eq:lin}).  In order to do that, it is enough to find a solution of
		\begin{equation}\label{transport1}
		\left.\begin{array}{rll}
		\pa_t v+\pa_x (c(x,t)v) &=&0\quad\quad\quad\mbox{for }(x,t)\in\re\times(0,T),\\
		v(x,0)&=& v_0(x)\quad\mbox{ for }x\in\re
		\end{array}\right\}	
		\end{equation}
		where $c(x,t)=g(u(x,t))$ and $u$ is the entropy solution of (\ref{eq:1scl}). We observed that $f^{\p}(u)(x,t)$ is Lipschitz in $x$--variable and so is $c(x,t)$ thanks to \eqref{defn_g}. We can solve (\ref{transport1}) by the method of characteristics and for that, we need to find the solution of the following Cauchy problem 
		\begin{equation}\label{ode}
		\frac{d}{dt}X(t,x_0)=c(X(t,x_0),t),\ \ X(0,x_0)=x_0,
		\end{equation}
		for each $x_0\in \re$. By using the classical Cauchy--Lipschitz theorem we obtain a unique solution of (\ref{ode}). In this way, we construct a solution of \eqref{transport1} for $L^\f$ initial data $v_0$. Let $v_0$ be defined as follows
		\begin{equation}
		v_0(x)=\left\{\begin{array}{ll}
		-1&\mbox{ if }2^{-2k}<x<2^{-2k+1}\mbox{ for }k\geq1,\\
		1&\mbox{ if }2^{-2k-1}<x<2^{-2k}\mbox{ for }k\geq1,\\
		1&\mbox{ if }x>1/2\mbox{ or }x<0.
		\end{array}\right.
		\end{equation}
	Consider the sequence $\{y_n\}$ defined as $y_n=(2^{-n}+2^{-n+1})/2$ for $n\geq1$. Now fix a $t\in[0,T]$. We define $z_n=X(t,y_n)$. Note that $v(z_n,t)=v_0(y_n)$. Let $s'\in(0,1]$. By the choice of $y_n$ we get
\begin{equation*}
TV^{s'}v(\cdot,t)\geq\sum\limits_{n=1}^{\f}\abs{v(z_n,t)-v(z_{n+1},t)}^{1/s'}=\sum\limits_{n=1}^{\f}\abs{v_0(y_n)-v_0(y_{n+1})}^{1/s'}=\f.
\end{equation*}
		 	Hence we obtain a solution $(u,v)$ of (\ref{eq:1scl}) and (\ref{eq:lin}) such that $u\notin BV^{s+\varepsilon}$ for any $\varepsilon > 0$ and $v\notin BV^{s'}$ for all $s'\in(0,1]$.
\end{proof}
 
 Now, the optimality in $BV^s$ for all time is presented.   In  a recent paper the existence of weak entropy solutions  for such triangular system are obtained  in  \cite{Triang} under the following assumptions for a convex flux. 
 \begin{enumerate}[label=(T-\arabic{*}), ref=(T-\arabic{*})]
 \item\label{T-1} The flux $f \in C^4$ is convex and $g \in C^3$.
 \item\label{T-2} Initial data $u_0$ belongs to $BV^{1/3}$  and $v_0$ to $L^\infty$.
 \item\label{T-3} The system is uniformly strictly hyperbolic,
 \begin{equation*}
          \inf_{|u|\leq M}  f'  >  \sup_{|u|\leq M}  g \,\mbox{ where }\,  M:=\|u_0\|_\infty.
 \end{equation*}
 \end{enumerate} 
Notice, that, in the Theorem \ref{theorem:sys}, the system is not assumed to be strictly hyperbolic, which $h=id$ for instance.  Here,  the strict hyperbolicity is assumed. Moreover, a minimal regularity of the initial data $u_0$ is needed to ensure the global existence in $L^\infty$ of a solution $(u,v)$.
 \begin{corollary} 
  Assume \ref{T-1}--\ref{T-3},  and  
  $$ s = \max \left(\frac{1}{3},\frac{1}{q} \right),$$
  where $q$ is the power of the degeneracy condition \eqref{decay_condition}. 
 There exists an initial data $u_0$ such that, for all $v_0 \in L^\infty$, the triangular system admits a global solution  
$u$ staying in $BV^s$ for all time, 
   $v \in L^\infty([0,+\infty),\R)$, 
   and,
   $\forall \epsilon >0, \forall t> 0, \, TV^{s+\epsilon} u(\cdot,t)=+\infty$.
 \end{corollary}

\section{The  multi-D Keyfitz-Kranzer system }\label{sec:KK-sys}

\quad In this section, we show that even for data with small total variation, renormalized solution to the  Keyfitz-Kranzer system may not be in $BV^s$. We use the example in \cite{Del-kk}. We modified the  renormalized solution considered in \cite{Del-kk} to show that even if the data has arbitrary small TV the $TV^s$ norm of the solution blows up. Here we mention the key points and the necessary changes. The rest follows from the analysis done in \cite{Del-kk}.  
\par Consider the following system 
\begin{eqnarray}\label{kk}
\pa_t u+\mbox{div}_z(h(|u|)u)&=&0\quad\quad\quad\mbox{for }z\in\re^m,\,t>0,\\
u(z,0)&=&u_0(z)\quad\mbox{ for }z\in\re^m.\nonumber
\end{eqnarray}
where $u:\re^m \times \re_+  \rr\re^k$ and $h\in C^{1}(\R,\R^m)$. Suppose $\eta:=|u|$ solves the following in the sense of Kru\v{z}kov
\begin{eqnarray}\label{scalar}
\pa_t \eta+\mbox{div}_z(h(\eta)\eta)&=&0\quad\quad\mbox{for }z\in\re^m,\,t>0,\\
\eta(\cdot,0)&=&|u_0|\quad\mbox{for }z\in\re^m.\nonumber
\end{eqnarray}
Let $\omega:=u/|u|$ solves the following transport equation
\begin{eqnarray}
\pa_t(\eta\omega)+\mbox{div}_z(h(\eta)\eta\omega)&=&0\quad\quad\quad\quad\mbox{for }z\in\re^m,\,t>0,\nonumber\\
\omega(\cdot,0)&=&u_0/|u_0|\quad\mbox{ for }z\in\re^m.\nonumber
\end{eqnarray}
We call $u=\eta\omega$ as \textit{renormalized entropy solution}. Note that the notion of renormalized entropy solution is different from the notion of standard entropy solution. Now we consider a special case of the system \eqref{kk} with $h=(g,0,\cdots,0)$. Then we have the following proposition
\begin{proposition} \label{pro:KK}
Let $h=(g,0,\cdots,0)$ for $g\in C^1(\R)$.	Let $k\geq2,m\geq2,$ and $b\in\re^k\setminus\{0\}$ such that $g^{\p}(|b|)\neq0$. Then there exists a sequence of initial data $u_0^n:\re^m\rr\re^k$ such that
	\begin{enumerate}
		\item $\|u_0^n-b\|_{BV(\re^m)}+\|u_0^n-b\|_{\f}\rr0$ as $n\rr\f$,
		\item $u_0^n=b$ on $\re^m\setminus B_\la(0)$ for some $\la>0$ independent of $n$,
		\item if $u^n$ is the renormalized entropy solution of (\ref{kk}) with initial data $u_0^n$  then $u^n(\cdot,t)\notin BV_{loc}^s$ for each $n\in\mathbb{N},\,t\in(0,1) $ and $s\in(0,1).$
	\end{enumerate} 
\end{proposition}
\begin{proof} 
	For the sake of simplicity, we prove  Proposition \ref{pro:KK} for $m=2$ and $k=2$. Suppose $g^{\p}(|b|)=1,\,g(|b|)=0$. Let $\de>0$ be fixed and $p=s^{-1}>1$. Let $m_i=i^{p+p\de}$ for all $i\in\N$. Then we have 
\begin{equation}
\sum\limits_{i}m_i2^{-i}<+\f.
\end{equation}
Let $\e>0$ be very small such that
\begin{itemize}
	\item $g$ is injective on $[|b|-2\e,|b|+2\e]$,
	\item $[-\e,\e]\subset g\left([|b|-2\e,|b|+2\e]\right)$.
\end{itemize} 
Then for sufficiently large $i$ we can choose $r_i\in[-2\e,2\e]$ such that $g(|b|+r_i)=2^{-i}$.  Note that for sufficiently large $i$ we have $r_i\leq 2^{-i+1}$. We write $\beta:=b/|b|$, and for each $i$ we choose a $\beta_i\in S^{k-1}$ such that $|\beta-\beta_i|=i^{-1-\de}$. Consider
\begin{eqnarray}
I_i&=&[2^{-i},2^{-i+1}),\\
I_i^j&=&\left[2^{-i}+\frac{(j-1)2^{-i}}{m_i},2^{-i}+\frac{j2^{-i}}{m_i}\right)\ \mbox{ for }1\leq j\leq m_i.
\end{eqnarray}
Define $\phi_i:\re^2\rr S^{k-1}$ as follows
\begin{equation}
\phi_i(x,y):=\left\{\begin{array}{lll}
\beta_i\ &\mbox{ when }y\in I_i \mbox{ and }[x2^i] \mbox{ is odd},&\\
\beta \ &\mbox{ otherwise}.&
\end{array}\right.
\end{equation}
Also we define $\Lambda_i:\re^2\rr\re$ as 
\begin{equation}
\Lambda_i(x,y):=\left\{\begin{array}{llllll}
r_i&\mbox{ when }&y\in I_i&\mbox{ for even }& j &\mbox{ and }x\in[-M,M],\\
r_{i+1}&\mbox{ when }&y\in I_i&\mbox{ for odd  }& j &\mbox{ and }x\in[-M,M],\\
0 &\mbox{ otherwise,}&&&&
\end{array}\right.
\end{equation}
where $M$ is some positive real number bigger than $1$ and it will be chosen later. Next we define
\begin{eqnarray}
\eta_0^n&:=&|b|+\sum\limits_{i=n}^\f\Lambda_i,\label{def:eta0}\\
\omega_0^n(x,y)&:=&\left\{\begin{array}{lll}
\phi_i(x,y)&\mbox{ when }y\in I_i\mbox{ for some $\geq n$ and }x\in[-M,M],&\\
\beta  &\mbox{ otherwise},&
\end{array}\right.\label{def:omega0}\\
u_0^n&:=&\eta_0^n\omega_0^n.\label{def:u0}
\end{eqnarray}
Next we show the following two properties of $\{u_0^n\}$ sequence. 
\begin{enumerate}[label=M.\arabic*]
	\item\label{M1} $\|u_0^n-b\|_{L^\f(\R^2)}\rr0$ as $n\rr\f$,
	\item\label{M2} $\|u_0^n-b\|_{BV(\re^2)}\rr0$ as $n\rr\f$.
\end{enumerate}
These have been shown in \cite{Del-kk}. For the sake of completion, we briefly mention key steps. From \eqref{def:eta0}--\eqref{def:u0} note that
\begin{equation}
\|u^n_0-b\|_{L^\f(\R^2)}\leq \abs{b}\abs{\B-\B_n}+r_n\leq\abs{b}n^{-1-\de}+2^{-n+1}\rr0\mbox{ as }n\rr\f.
\end{equation}
We observe that $supp(u^n_0)\subset [-M,M]\times[0,1]\subset[-M,M]^2$ as $M>1$. Note that $\|\eta_0^n\|_{L^{\f}(\R^2)}\leq \abs{b}+1$ and $\|\omega_0^n\|_{L^{\f}(\R^2)}\leq 2\abs{\B}+1$ for all $n\in\N$. Therefore, to prove (\ref{M2}) it is enough to show that \begin{equation}\label{convergence:rho-omega}
\|\eta_0^n-|b|\|_{BV([-2M,2M]^2)}\rr0\mbox{ and }\|\omega_0^n-\beta\|_{BV([-2M,2M]^2)}\rr0\mbox{ as }n\rr\f.
\end{equation} 
Similar to \cite{Del-kk}, we can show
\begin{align*}
\|\eta_0^n-\abs{b}\|_{BV([-2M,2M]^2)}&\leq 4M^2\|u_0^n-b\|_{L^{\f}(\R^2)}+\sum\limits_{i\geq n}m_i2^{-i}+(4M+2)r_n,\\
\|\omega_0^n-\B\|_{BV([-2M,2M]^2)}&\leq 4M^2\|\omega_0^n-\B\|_{L^{\f}(\R^2)}+2M\sum\limits_{i\geq n}2^{-i}i^{-1-\de}2^i\\
&+2M\sum\limits_{i\geq n}[i^{-1-\de}+(i+1)^{-1-\de}]+(4M+2)r_n.
\end{align*}
Note that $\|\omega_0^n-\B\|_{L^{\f}(\R)}\leq \abs{\B-\B_n}\rr0$ as $n\rr\f$. Since $\sum\limits_{i\geq 1}m_i2^{-i}<\f$ and $\sum\limits_{i\geq 1}i^{-1-\de}<\f$, we have
\begin{equation*}
\sum\limits_{i\geq n}m_i2^{-i}\rr0,\,\sum\limits_{i\geq n}i^{-1-\de}\rr0\mbox{ and }\sum\limits_{i\geq n}[i^{-1-\de}+(i+1)^{-1-\de}]\rr0\mbox{ as }n\rr\f.
\end{equation*}
Hence, we obtain \eqref{convergence:rho-omega}. Suppose $u^n$ is the unique renormalized solution of (\ref{kk}). We have seen $\eta^n$ is the unique solution to (\ref{scalar}) with initial data $\eta_0^n$. Notice that $\eta_0^n(\cdot,y)$ is constant on $[-M,M]$ and by finite speed of conservation laws we get $\eta^n(x,y,t)=\eta_0^n(x,y)$ if $(x,y,t)\in\left\{\sqrt{x^2+y^2}\leq C(M-t)\right\}$ where $C=C(\|\eta_0^n\|_{\f})$. Note that for each $R>0$ we can choose $M>0$ large enough such that
\begin{equation}
\eta^n(x,y,t)=\eta_0^n(x,y)\mbox{ for $t\in[0,1]$ and $(x,y)\in(-R,R)\times[0,1]$}.
\end{equation}
We will choose $R$ later. To analyze the angular part $\omega^n:=u^n/|u^n|$ we use the fact that $\eta^n$ is a constant on the curve $\Psi_n(\cdot,x,\cdot)$ where $\Psi_n(\cdot,x,\cdot)$ satisfies 
\begin{eqnarray}
\frac{d}{dt}\Psi_n(t,x,y)&=&h(\eta^n(\Psi_n(t,x,y),t)),\\
\Psi_n(0,x,y)&=&(x,y).
\end{eqnarray}
We can choose $R$ large enough so that for any $(\tau,x_1,y_1)\in[0,1]^3\subset[0,1]\times[-R,R]\times[0,1]$, the curve $t\mapsto \Psi_n(t,x_0,y_0)$ lies on the plane $y=y_1$ for $t\in(0,\tau)$ and remains a straight line for $t\in(0,\tau)$ (see \cite{Del-kk} for more detailed discussion on this). As it has been observed in \cite{Del-kk}, choice of $R$ can depend only on $g$ and $\|\rho_0^n\|_{L^{\f}(\R^2)}$. Since there exists a constant $C>0$ such that $\|\rho_0^n\|_{L^{\f}(\R^2)}\leq C$ for all $n\geq1$, we conclude that choice of $R$ does not depend on $n$. Once we fix the choice of $R$, we make the choice of $M$. By a similar discussion as in \cite{Del-kk} we have the following,
\begin{itemize}
	\item if $\eta_0^n(x,y)=|b|$, then $\omega^n(x,y,t)=\omega_0^n(x,y)$,
	\item if $\eta_0^n(x,y)=|b|+r_i$, then $\omega^n(x,y,t)=\omega_0^n(x-t2^{-i},y)$.
\end{itemize}
Therefore, for $j\in\{1,\cdots,m_i\},i\geq n$ and $l\in\{1,\cdots,2^i-1\}$ the function $\omega^n(\cdot,\cdot,t)$ jumps on the segments
\begin{equation*}
J_{j,i,l}:=\left\{y=2^{-i}+\frac{j2^{-i}}{m_i},x\in[l2^{-i},(l+t)2^{-i}]\right\}.
\end{equation*}
 For a fixed $t>0$, suppose $\omega^n(\cdot,\cdot,t)\in BV^{s}_{loc}(\R^2)$. Then, by \eqref{def:BVs:loc} there exists a $W^n\in BV_{loc}(\R^2)$ and $\pi^n\in Lip^s(\R)$ such that $\omega^n(x,y,t)=\pi^n\circ W^n(x,y)$. Hence, $\abs{\omega^n(x_1,y_1,t)-\omega^n(x_2,y_2,t)}\leq C_1\abs{W^n(x_1,y_1)-W^n(x_2,y_2)}^{s}$ where $C_1=Lip^s(\pi^n)$. Let $p=s^{-1}$. Then $\abs{W^n(x_1,y_1)-W^n(x_2,y_2)}\geq C_1^{-p}\abs{\omega^n(x_1,y_1,t)-\omega^n(x_2,y_2,t)}^p$. Now we observe that
\begin{equation*}
V_{i}:=\int\limits_{J_{j,i,l}}\abs{W^n(x,y+)-W^n(x,y-)}\, d\mathcal{H}^{1}(x)\geq C_1^{-p}t2^{-i}\abs{\B-\B_i}^p=2^{-i}C_1^{-p}ti^{-p-p\de},
\end{equation*}
where $\mathcal{H}^1$ denotes the one dimensional Hausdorff measure. Therefore, we have
\begin{equation*}
\|W^n(\cdot,\cdot)\|_{BV([-2M,2M]^2)}\geq\sum\limits_{i\geq n}\sum\limits_{j=1}^{m_i-1}\sum\limits_{l=1}^{2^i-1}V_{i} \geq\sum\limits_{i\geq n}\sum\limits_{j=1}^{m_i-1}\sum\limits_{l=1}^{2^i-1}t2^{-i}i^{-p-p\de}\geq\frac{t}{2}\sum\limits_{i\geq n}(m_i-1)i^{-p-p\de}.
\end{equation*}
Since $m_i=i^{p+p\de}$ we obtain $\|W^n(\cdot,\cdot)\|_{BV([-2M,2M]^2)}=\f$. This gives a contradiction. Hence, $\omega^n(\cdot,\cdot,t)\notin BV^s_{loc}(\R^2)$.
\end{proof}

	\noi\textbf{Acknowledgement.} Authors thank the anonymous referee for valuable comments and suggestions. The first author would like to thank Inspire faculty-research grant  DST/INSPIRE/04/2016/0\linebreak00237. Authors would like to thank the IFCAM project ``Conservation laws: $BV^s$, interface and control". The first and third authors acknowledge the support of the Department of Atomic Energy, Government of India, under project no. 12-R\&D-TFR-5.01-0520. The second and  the fourth authors would like to thank TIFR-CAM for the hospitality. 




\end{document}